\newcounter{lil}

\newcounter{lil22}
\newenvironment{steps-2}
{\em \begin{list} {  (\roman{lil22})} {\usecounter{lil22} \em
\setlength{\leftmargin}{1.3cm}
\setlength{\topsep}{0.2cm} \setlength{\itemsep}{0.0cm}
\setlength{\parsep}{0.1cm} \setlength{\itemindent}{0.4cm}
\setlength{\parskip}{0.0cm}}} {\end{list}}

\newcounter{gr11}
\newenvironment{letters}
{\begin{list} { (\alph{gr11})$\, $} {\usecounter{gr11}
\setlength{\labelwidth}{-0.2cm} \setlength{\leftmargin}{0.4cm}
\setlength{\topsep}{0.1cm} \setlength{\itemsep}{0.2cm}
\setlength{\parsep}{0.1cm} \setlength{\itemindent}{0.4cm}
\setlength{\parskip}{0.0cm}}} {\end{list}}

\newcommand{\AAA}{{U}}

\newcommand{\Law}{{\mathcal{L}aw}}

\newcommand{\CCC}{H}
\newcommand{\mg}{\mathfrak{g}}
\newcommand{\baray}{\begin{array}{rcl}}
\newcommand{\earay}{\end{array}}
\newcommand{\barray}{\begin{array}{rcl}}
\newcommand{\earray}{\end{array}}

\newcommand{\bcal}{\mathcal{B}}

\newcommand{\mF}{\mathfrak{F}}
\newcommand{\mT}{\mathfrak{T}}
\newcommand{\mS}{\mathfrak{S}}

\newcommand{\sn}{{s_0}}
\newcommand{\sr}{{0}}

\newcommand{\levy}{L\'evy }

\newcommand\coc[1]{{\overline{#1}}} 

\newcommand\dela[1]{}
\newcommand\noot[1]{}

\newcommand{\bcase}{\begin{cases}}
\newcommand{\ecase}{\end{cases}}

\newcommand\cadlag{c{\`a}dl{\`a}g }

\newcommand\del[1]{}

\def\bibiter#1{\bibitem{#1}}

\def\bibiter#1{\bibitem{#1}}                               

\newcommand{\lk}{\left}
\newcommand{\lqq}{\lefteqn}
\newcommand{\rk}{\right}
\newcommand{\la}{{\langle}}
\newcommand{\ra}{{\rangle}}

\newcommand{\LL}{{\rm I \kern -0.2em L}}

\newcommand{\ep} {\varepsilon }

\newcommand{\be} {\begin{enumerate} }
\newcommand{\ee} {\end{enumerate} }

\newcommand{\CE}{{{ \mathcal E }}}
\newcommand{\CC}{{{ \mathcal C }}}

\newcommand{\CZ}{{{ \mathcal Z }}}
\newcommand{\CT}{{{ \mathcal T }}}

\newcommand{\CA}{{{ \mathcal A }}}

\newcommand{\CH}{{{ \mathcal H }}}

\newcommand{\CB}{{{ \mathcal B }}}

\newcommand{\CM}{{{ \mathcal M }}}

\newcommand{\BF}{{{ \mathbb{F} }}}
\newcommand{\CF}{{{ \mathcal F }}}

\newcommand{\CL}{{{ \mathcal L }}}

\newcommand{\RR}{{\mathbb{R}}}

\newcommand{\DD}{\mathbb{D}}

\newcommand{\NN}{\mathbb{N}} 

\newcommand{\PP}{{\mathbb{P}}}

\newcommand{\EE}{ \mathbb{E} }
\newcommand{\TT}{{\rm I \kern -0.2em T}}

\newcommand{\DEQS}{\begin{eqnarray*}}
\newcommand{\EEQS}{\end{eqnarray*}}
\newcommand{\DEQSZ}{\begin{eqnarray}}
\newcommand{\EEQSZ}{\end{eqnarray}}
\newcommand{\DEQ}{\begin{eqnarray}}
\newcommand{\EEQ}{\end{eqnarray}}

\newcommand{\hh}{h}
\newcommand{\Jeh}{J_\ep^\frac12}

\documentclass[a4paper,onesided]{amsart}
\usepackage[left=2.25cm,right=2.25cm,top=4cm,bottom=4cm]{geometry}
\usepackage{color}
\usepackage{enumerate}
\usepackage{upgreek}
\usepackage{color}
\usepackage{amsmath}   
\usepackage{amsfonts,amssymb, color}
\usepackage{latexsym}
\input colordvi

\usepackage{amsfonts}

\theoremstyle{plain}

\numberwithin{equation}{section}

\newcounter{gr1}

\newcounter{gr1n}
\newenvironment{numlistn}
{\begin{list} { (\roman{gr1n})} {\usecounter{gr1n}
\setlength{\leftmargin}{0.9cm}
\setlength{\topsep}{0.1cm} \setlength{\itemsep}{0.0cm}
\setlength{\parsep}{0.1cm} \setlength{\itemindent}{-0.7cm}
\setlength{\parskip}{0.0cm}}} {\end{list}}

\theoremstyle{plain}

\numberwithin{equation}{section}

\begin{document}

\title{The nonlinear Schr\"odinger Equation driven by jump processes}
\thanks{This work was supported by the FWF-Project
P17273-N12 and the ANR project Stosymap (ANR 2011-B501-015-03)}

\author{Anne de Bouard \and Erika Hausenblas}
\address{A. de Bouard, CMAP, Ecole polytechnique, CNRS, Universit\'e Paris-Saclay, 91128 Palaiseau, France}
\address{E. Hausenblas, Lehrstuhl Angewandte Mathematik, Montanuniversitat Leoben, Franz-Josef-Strasse 18,
8700 Leoben, Austria}

\date{\today}

\begin{abstract}
The main result of the paper is the existence  of a solution of the nonlinear Schr\"odinger equation with a \levy noise with infinite activity.
To be more precise, let $A=\Delta$ be the Laplace operator with $D(A)=\{ u\in L ^2 (\RR ^d): \Delta u \in L ^2 (\RR ^d)\}$.
Let $Z\hookrightarrow L ^2(\RR ^d)$ be a function  space and $\eta$ be a Poisson random measure on $Z$,
let $g:\RR\to\mathbb{C}$ and $\hh:\RR\to\mathbb{C}$ be some given functions, satisfying certain conditions specified later.
Let $\alpha\ge 1$ and $\lambda\ge 0$. We are interested in the solution of the following equation
\DEQSZ\label{itoeq}
\hspace{2cm} 
\lk\{ \baray \lqq{ i \, d u(t,x)  -  \Delta u(t,x)\,dt +\lambda |u(t,x)|^{\alpha-1} u(t,x) \, dt\hspace{2cm}}&&
\\&=& \int_Z u(t,x)\, g(z(x))\,\tilde \eta (dz,dt)+\int_Z u(t,x)\, \hh (z(x))\, \gamma (dz, dt),\\
u(0)&=& u_0. \earay \rk.
 \EEQSZ
First we consider the case, where the \levy process is a compound Poisson process. With the help of this result
we can tackle the general case, and show that \eqref{itoeq} has a solution.
\end{abstract}

\maketitle



\newtheorem{theorem}{Theorem}[section]
\newtheorem{notation}{Notation}[section]
\newtheorem{claim}{Claim}[section]
\newtheorem{lemma}[theorem]{Lemma}
\newtheorem{corollary}[theorem]{Corollary}
\newtheorem{example}[theorem]{Example}
\newtheorem{assumption}[theorem]{Assumption}
\newtheorem{tlemma}{Technical Lemma}[section]
\newtheorem{definition}[theorem]{Definition}
\newtheorem{remark}[theorem]{Remark}
\newtheorem{hypotheses}{H}
\newtheorem{hypo}{Hypothesis}
\newtheorem{proposition}[theorem]{Proposition}
\newtheorem{Notation}{Notation}
\renewcommand{\theNotation}{}

\renewcommand{\labelenumi}{\alph{enumi}.)}





\section{Introduction}\label{sec_intro}

We consider in the present paper the problem of existence of solutions for the nonlinear Schr\"odinger equation with \levy noise.
To be more precise, let $A=\Delta$ be the Laplace operator with $D(A)=\{ u\in L ^2 (\RR ^d): \Delta u \in L ^2 (\RR ^d)\}$.
Let $Z\hookrightarrow L ^2(\RR ^d)$ be a function  space and $\eta$ be a Poisson random measure on $Z$,
let $g:\RR\to\mathbb{C}$ and $\hh:\RR\to\mathbb{C}$ be some given functions, satisfying certain conditions specified later.
Let $\alpha\ge 1$ and $\lambda\in\RR$. We are interested in the following equation
\DEQSZ\label{itoeq_1}
\hspace{2cm}
 \lk\{ \baray \lqq{ i \, d u(t,x)  - A u(t,x)\,dt +\lambda |u(t,x)|^{\alpha-1} u(t,x) \, dt\hspace{2cm}}&&
\\&=& \int_Z u(t,x)\, g(z(x))\,\tilde \eta (dz,dt)+\int_Z u(t,x)\, \hh (z(x))\, \gamma (dz, dt),\\
u(0)&=& u_0. \earay \rk.
 \EEQSZ
Our aim is to investigate the conditions on the nonlinearity, on the space $Z$ and on the complex valued functions $c$ and $g$, under
which  there exists a weak or martingale solution to \eqref{itoeq_1}.

The Nonlinear Schr\"odinger equation  (NLS)  is a universal model that describes the propagation of nonlinear waves in
dispersive media. It may e.g. appear  as  a so-called  modulation  equation, describing the complex enveloppe of a highly
oscillating field in nonlinear optics, and in particular in fiber optics (see \cite{agrawal, newell}).
It may also be derived from the water wave problem, thanks to scaling and perturbation arguments,
to describe the propagation of surface waves of finite amplitude in deep fluids (see \cite{wayn,zak}).
The propagation of nonlinear dispersive waves in nonhomogeneous or
random media (or taking account of temperature effects) can be modelled by the nonlinear equation
with a random force, or a random potential (see e.g. \cite{abdullaev, bang, falkovich}).

%

When the stochastic perturbation is a Wiener process, the equation is well treated and existence and uniqueness of the solution is known,
under reasonable assumptions on the noise correlation and on the nonlinearity.
For more information see  \cite{barbu,anne2,anne1,MR2135313,MR2652190}. The case where the nonlinear Schr\"odinger equation is
perturbed by a \levy process
is much less treated in the literature. In \cite{villa2,villa1}, the authors consider the NLS equation with randomly distributed,
but isolated jumps. In the context of fiber optics, the model would describe random amplification of the signal at random
(but isolated) locations along the fiber (see \cite{kodama}). In that situation the existence and uniqueness of solutions is easily deduced
from the classical results known in the deterministic case, and the motivations in \cite{villa2, villa1} were to obtain the evolution law
of some physical observables of the solution.

Here, we consider the more general case where the noise is an infinite dimensional \levy process, with possibly non isolated jumps, and
we investigate the existence of martingale solutions. Before stating the precise result, let us introduce some notations.

\begin{notation}
\noot{We denote by $B_{p,q}^ s(\RR^d)$ the natural Besov spaces defined in \cite[Chapter 1.2.5, p.\ 8]{triebelII} or \cite[Definition 2, p.\ 8]{runst}.}
For $k\in\NN_0$ we denote by $H^k_p(\RR^d)$ the classical Sobolev spaces defined in \cite[Chapter 3, Definition 3.1]{triebelharoske}.
For $\delta\ge 0$ and $p\in(0,\infty]$ let $L^p_\delta(\RR^d)=\{ v\in L^p(\RR^d )$ with $\int_{\RR^d } (1+|x|^{2})^\frac \delta2\, |u(x)|^p\, dx<\infty\}$.

For any index $p\in[1,\infty]$ we denote throughout the paper the conjugate element by $p'$. In particular, we have $\tfrac 1p + \tfrac 1 {p'}=1$.
For complex valued functions $u$ and $v$ in $L^2(\RR^d)$, we denote by $\langle u,v\rangle$ the (real) inner product
$$
 \langle u,v\rangle =\Re \int_{\RR^d} u(x) \overline{v(x)} \, dx.
$$
Given a Banach space $E$ and a number $R>0$, we denote by $B_E(R)$ all elements with norm smaller or equal to $R>0$, i.e.\ $B_E(R) :=\{ x\in E, |x|_E\le R\}$.

Suppose that  $(Z,{\mathcal{Z}})$ is a measurable space.  By  $M(Z)$, respectively $M_+(Z)$,  we will denote the
set of all $\mathbb{R}$, respectively $[0,\infty]$-valued  measures on $(Z,{\mathcal{Z}})$.  By ${\mathcal{M}}(Z)$, respectively ${\mathcal{M}}_+(Z)$, we will denote
the $\sigma$-field on $M(Z)$, respectively $M_+(Z)$, generated by functions
\[i_B:M(Z) \ni\mu \mapsto \mu(B)\in {\mathbb{R}},\]
respectively by functions
\[i_B:M_+(Z) \ni\mu \mapsto \mu(B)\in [0,\infty],\]
for all $B\in {\mathcal{Z}}$.
 Similarly,
by $M_I( Z)$ we will  denote the family of all
$\overline{\mathbb{N}}$-valued measures on $(Z,{\mathcal{Z}})$ $(\overline{\mathbb{N}}=\NN\cup\{\infty\})$,   and  by
${\mathcal{M}}_I(Z)$ the $\sigma$-field
 on $M_I(Z)$ generated by functions
$i_B:M(Z) \ni\mu \mapsto \mu(B)\in \overline{\mathbb{N}}$, $B\in {\mathcal{Z}}$.

Finally,    by
${\mathcal{Z}}\otimes \mathcal{B}({\mathbb{R}_+})$  we  denote  the product
$\sigma$-field on $Z\times \mathbb{R}_+$ and by $\nu\otimes
\lambda$ we denote the product measure of $\nu$ and the Lebesgue measure
$\lambda$.

\end{notation}

\section{Preliminaries and main result}


%
%

%
{Throughout the whole paper, we assume that
$\mathfrak{A}=(\Omega,\CF,\BF,\PP)$ is a complete filtered probability space with right continuous filtration
 $\{\CF_t\}_{t\ge
0}$, denoted by  by $\BF$.
The following definitions are presented here for the sake of
completeness because the notion of time homogeneous random measure
is introduced in many, not always equivalent ways.
\begin{definition}\label{def-Prm}(see \cite{ikeda}, Def. I.8.1)
Let $(Z,\CZ)$  be a measurable space. \del{ and let $(\Omega,\CF,\BF,\PP)$ be a 
probability space with filtration $\BF=\{\CF_t\}_{t\ge 0}$.}\\ A
{\sl Poisson random measure} $\eta$ on $(Z,\CZ)$   over
$(\Omega,\CF,\BF,\PP)$  is a measurable function $\eta:
(\Omega,\CF)\to (M_I(Z\times \RR_+),\CM_I(Z\times \RR_+)) $, such
that {\begin{trivlist} \item[(i)] for each $B\in  \CZ \otimes
\mathcal{B}({\mathbb{R}_+}) $,
 $\eta(B):=i_B\circ \eta : \Omega\to \bar{\mathbb{N}} $ is a Poisson random variable with parameter\footnote{If  $\EE \eta(B) = \infty$, then obviously $\eta(B)=\infty$ a.s..} $\EE\eta(B)$;
\item[(ii)] $\eta$ is independently scattered, i.e. if the sets $
B_j \in   \CZ\otimes \mathcal{B}({\mathbb{R}_+})$, $j=1,\cdots,
n$, are  disjoint,   then the random variables $\eta(B_j)$,
$j=1,\cdots,n $, are independent;
\item[(iii)] for each $U\in \CZ$, the $\bar{\mathbb{N}}$-valued
process $(N(t,U))_{t\ge 0}$  defined by
$$N(t,U):= \eta(U \times (0,t]), \;\; t\ge 0$$
is $\BF$-adapted and its increments are independent of the past,
i.e.\ if $t>s\geq 0$, then $N(t,U)-N(s,U)=\eta(U \times (s,t])$ is
independent of $\mathcal{F}_s$.
\end{trivlist}
}

\end{definition}

\begin{definition}
The compensator of a random measure $\eta$ on a Banach space $Z$
is the unique predictable measure $\gamma : \CZ \times \bcal ({\mathbb{R}}^0_{+}) \to {\mathbb{R}}$, such that for any $A\in\CZ$ the process
$$
\RR_+^0 \ni t \mapsto \eta(A\times [0,t])-\gamma(A\times [0,t])
$$
is a martingale over $\mathfrak{A}$.
We will denote by $\tilde{\eta }$ the \it  compensated Poisson random measure \rm  defined by
$\tilde{\eta }: = \eta - \gamma $.
\end{definition}

\begin{remark}
Assume that $\eta$ is a time homogeneous Poisson random measure
on  $(Z,\CZ)$ over $(\Omega,\CF,\BF,\PP)$. It
turns out that the compensator $\gamma$ of $\eta$ is uniquely
determined and moreover
$$
\gamma: \CZ \times \CB(\RR^+)\ni (A,I)\mapsto  \nu(A)\times
\lambda(I),
$$
where the $\sigma$--finite measure $\nu:\CZ\to \RR_+\cup\{\infty\}$ is defined by  $\CZ\ni A \mapsto \nu(A):=\EE\eta(A\times [0,1])$. 
The  difference between a time homogeneous  Poisson random measure
$\eta$  and its compensator $\gamma$, i.e.  $\tilde
\eta=\eta-\gamma$, is called a  {\em compensated Poisson random
measure}. The measure $\nu$ is called {\sl intensity measure} of $\eta$.
\end{remark}

Let  $Z\hookrightarrow L^ 2(\RR^ d)$ be a function space, $\nu$ a $\sigma$--finite measure on $Z$ such that
$$\nu(\{0\})=0, \quad \int_Z(|z|^2 \wedge 1)\nu(dz)<\infty, \quad \mathrm{ and }\;
\nu(Z\setminus B_Z(\ep))<\infty \; \mathrm{for \; all }\; \ep>0. $$
Let $\eta$ be a time homogenous Poisson random measure on $Z$ with intensity measure $\nu$ over
$\mathfrak{A}$.

Let $g:\RR\to \mathbb{C}$ and $\hh:\RR \to \mathbb{C}$ be two functions specified later.
We will denote by $G:L ^ 2(\RR^ d )\to L^ 2 (\RR^ d )$ and $\CCC:L ^ 2(\RR^ d )\to L^ 2 (\RR^ d )$ the Nemytskii operators
associated to the functions $g$ and $\hh$, and defined by
$$
(G(z)) (x) :=g(z(x)),\quad 
\mbox{and}\quad  (\CCC(z)) (x) :=h(z(x)),\quad z\in Z,\,x\in\RR^ d.
$$

\medskip
We are now interested in the following equation
\DEQSZ\label{itoeqn}
\hspace{3cm}
 \lk\{ \baray \lqq{ i \, d u(t,x)  -  A u(t,x)\,dt +\lambda|u(t,x)|^{\alpha-1} u(t,x) \, dt\hspace{2cm}}&&
\\&=& \int_Z u(t,x)\, g(z(x))\,\tilde \eta (dz,dt)+\int_Z u(t,x)\, \hh(z(x))\, \gamma (dz, dt),\\
u(0)&=& u_0, \earay \rk.
 \EEQSZ

Let us denote by $(\CT(t))_{t\ge 0}$  the group of isometries generated by the operator $-iA$. 
As is classical in the framework of evolution equations, we will consider a
mild solution of equation \eqref{itoeqn}, whose definition is given below.

\begin{definition}
Let $E$ 
be a  Banach space.
We call $u$ an $E$--valued solution to Equation \eqref{itoeqn}, if and only if  $u\in\DD(0,T;E)$ $\PP$-a.s.,
the terms
$$
\int_0^ t \CT(t-s)
|u(s)|^{\alpha-1} u(s)\, ds,\quad  
 \int_0^t\int_Z  \lk|\CT(t-s) u(s)\, G(z)\rk|^ 2 \nu(dz)\, ds
$$
and
$$
 \int_0^t\int_Z  \CT(t-s) u(s)\, H(z)  \nu(dz) \, ds,
$$
are well defined for any $t\in[0,T]$ in $E$ and $u$ solves $\PP$-a.s.\ the integral equation
\DEQSZ
\label{eq_1}
u(t)  & = & \CT(t) u_0 + i \lambda \int_0^ t \CT(t-s)
|u(s)|^{\alpha-1} u(s)\, ds \\
 \nonumber
&& -i\int_0^t\int_Z  \CT(t-s) u(s)\, G(z) \,\tilde \eta (dz,ds)
-i \int_0^ t \int_Z \CT(t-s) {u(s) \CCC(z) \gamma(dz,ds) }
.
\EEQSZ
\end{definition}

However, for a \levy noise with infinite activity, we could not show the existence of a unique
strong solution, only the existence of a martingale solution. A concept, defined
in the following.

\begin{definition}\label{Def:mart-sol}
Let $E$ be a  separable  Banach space.
Let $(Z,\CZ)$ be a measurable space and $\nu$ a $\sigma$--finite measure on $(Z,\CZ)$.
Suppose that  $G$ and $H$ are a densely defined function from $ Z$
to $E$.
Let $u_0\in E$.
\\
A {\sl martingale solution} on $E$  to the Problem  \eqref{itoeqn}  is a system
\begin{equation}
\left(\Omega ,{\CF},{\Bbb P},\BF, 
\{\eta(t)\}_{t\ge 0},\{u(t)\}_{t\ge 0}\right)
\label{mart-system}
\end{equation}
such that
\begin{numlistn}
\item  $(\Omega ,{\CF},\BF,{\Bbb P})$ is a complete filtered
probability
space with filtration $\BF=\{{\CF}_t\}_{t\ge 0}$,
\item
${\{\eta(t)\}_{t\ge 0}}$  is a time homogeneous Poisson
random measure on $(Z,\CB(Z))$ over $(\Omega
,{\CF},\BF,{\Bbb P})$ with intensity measure $\nu$,
\item $u=\{u(t)\}_{t\ge 0}$  is a $E$--valued mild solution  to  the Problem \eqref{itoeqn}.
\end{numlistn}
%
\end{definition}
In order to be able to show the existence of a solution, the space $Z$,
the \levy measure $\nu$ and the functions $g,h:\RR^d \to\mathbb{C}$ have to satisfy certain conditions.
In particular, they have to satisfy the following hypothesis.
\begin{hypo}\label{hypogeneral}
\renewcommand{\theenumi}{\roman{enumi}}
\renewcommand{\labelenumi}{(\theenumi)\,}
First, we assume that $Z$ a function space and $\nu$ a \levy measure on $Z$ such that
\begin{enumerate}
  \item $Z$ is continuously embedded in the Sobolev space $ W^ 1_{\infty}(\RR^d)$; 
  \item the \levy measure $\nu$ satisfies the following integrability conditions
  \begin{enumerate}
    \item $C_0(\nu):=\int_Z |z|_{L^\infty}^ 2 \nu(dz)<\infty;\phantom{\Big|}$

    \item $C_1(\nu):=\int_Z |z|_{H^ 1 _\infty}^ 2 \nu(dz)<\infty;\phantom{\Big|}$
    \item  $C_2(\nu):=\int_Z \sup_{x\in {\RR^d}} |x|^ 2 |z(x)|^ 2 \nu(dz)<\infty;\phantom{\Big|}$
    \item $C_3(\nu):=\int_Z |z|_{L^\infty}^ 4 \nu(dz)<\infty.\phantom{\Big|}$
  \end{enumerate}
\end{enumerate}

In addition the functions $g:\RR\to\mathbb{C}$ and $\hh :\RR\to\mathbb{C}$ are satisfying the following items:
\begin{enumerate}
\renewcommand{\theenumi}{\roman{enumi}}
\setcounter{enumi}{2}
\item $g$, $\hh $ and their first order derivatives are of linear growth, i.e.\ there exist some constants $C_g$ and $C_\hh $ such that
$$|g(\xi )|, | g'(\xi )|\le C_g |\xi | \quad \mbox {and} \quad  |\hh (\xi )|,|\hh ' (\xi )|\le C_\hh  |\xi |, \quad \forall \xi \in \RR.
$$
  \item $g(0)=0$ and $\hh (0)=0$;
\end{enumerate}
\end{hypo}

\begin{remark}
Hypothesis \ref{hypogeneral} implies that the Nemitskii operators $G$ and $\CCC$ associated to $g$ and $\hh $ map $L ^ \infty(\RR^ d)$ into  $L ^ \infty(\RR^ d)$,
and their Frechet derivative $\nabla G$ and $\nabla \CCC$ also map $L ^ \infty(\RR^ d)$ into  $L ^ \infty(\RR^ d)$.
\end{remark}

To show the existence of the solution to the nonlinear Schr\"odinger equation with \levy noise of infinite activity,
we use the  technical Lemma \ref{local_ex} below, which gives existence and uniqueness of the solution to \eqref{itoeqn}
where the Levy process is  a compound Poisson Process, i.e.\ if
the \levy process has only finite activity.
Then, we use a cut off of the small jumps with a cut off parameter $\ep>0$ in order to get a noise with finite activity,
and we apply Lemma  \ref{local_ex} to get  the existence of a unique solution of  \eqref{itoeqn} with the cut-off noise of finite activity.
 In a second step,
we  show the existence of a limit as $\ep\to0$.
Here, uniform bounds on the $L^ 2(\RR^d) $ norm and $H^ 1_2(\RR^d)$ norm play an important role.
%
Similarly to the deterministic setting, this uniform bounds are obtained by controlling the mass and energy in average.
Let us define the mass by
\DEQSZ\label{mass}
\mathcal{E}(u):= \int_{\RR^d} | u(x)|^ 2 \, dx,
\EEQSZ
and the energy by 
\DEQSZ\label{hamiltonian}
\\
\nonumber
\mathcal{H}(u):= \frac12 \int_{\RR^d} |\nabla u(x)|^ 2 \, dx + {\lambda\over \alpha+1} 
\int_{\RR^d } |u(x)|^ {\alpha-1} { u(x)\, \coc{u(x)}\,} dx.
\EEQSZ

Under certain constrains on $g$, $\hh$ and $u_0$, the mass will be conserved.
For this purpose, we introduce  following hypothesis:

\begin{hypo}\label{hypoas}
\renewcommand{\theenumi}{\roman{enumi}}
\renewcommand{\labelenumi}{(\theenumi)\,}

Let us assume
\begin{enumerate}
  \item
  $\Im(g(\xi))=\Im(\hh (\xi)), \quad \xi\in\RR$;
  \vskip 0.07 in
  \item
  $|1{-i}g(\xi)|=1, \quad \xi\in\RR.$
\end{enumerate}
\end{hypo}

\begin{hypo}\label{hypomean}
Let us assume
  \item
  \DEQS
   2\Im(h(\xi))+ |g(\xi)|^ 2 =0,\quad \xi\in\RR.
  \EEQS
\del{  \DEQSZ\label{cond11}
  \Im(g(\xi))=0, \quad \xi\in\RR,
  \EEQSZ
  \item
  \DEQSZ\label{cond21}
  \Im \hh (\xi) = g(\xi) ^2, \quad \xi\in\RR.
  \EEQSZ
}
\end{hypo}

\del{\begin{hypo}\label{hypogrown}
Let us assume
$$C_3(\nu):=\int_Z |z|_{L^\infty}^ 4 \nu(dz)<\infty.\phantom{\Big|};
$$
\end{hypo}}

In fact, Hypothesis \ref{hypogeneral} gives only conditions under which the solution exists.
Hypothesis \ref{hypoas} and Hypothesis \ref{hypomean} give the conditions for the conservation of $\CE(u)$ $\PP$--a.s.\ or in mean.
We now state our main result.
\begin{theorem}\label{main}
Let  $\eta$ be  a time homogenous Poisson random measure on a Banach space $Z$ with \levy measure $\nu$ satisfying Hypothesis \ref{hypogeneral}.
Assume $\lambda>0$, $1\le \alpha<1+4/(d-2)$ if $d>2$, or $1\le \alpha<+\infty$ if $d=1,2$. Let $u_0\in H ^1_2(\RR ^d)$ with
$$ \int_{\RR^d} x  ^2 |u_0(x)| ^2 \, dx <\infty,
$$
then for any $\gamma<1$ there exists a $H^\gamma_2(\RR^d)$--valued martingale mild solution to \eqref{itoeqn}, such that 
\renewcommand{\labelenumi}{(\arabic{enumi})}
\begin{enumerate}
  \item 
  for any $T>0$, there exists a  constant $C=C(T,C_0(\nu),C_3(\nu),C_g,C_h)>0$ such that
$$   \EE\sup_{0\le s\le T}  |u(s)|^2 _{L^2 }\le C\, (1+ |u_0|^2 _{L^2 }).$$
  \item for any $T>0$, there exists a constant  $C=C(T,C_0(\nu),C_1(\nu),C_g,C_h)>0$ such that
$$
\EE \sup_{0\le t\le T} \CH(u(t)) \leq C\, \lk( 1+ \CH(u_0) \rk).
$$

  \item for any $T>0$, there exists a constant $C=C(T, C_2(\nu),C_g,C_h)>0$ such that
$$
\EE \sup_{0\le t\le T} \int_{\RR^d } x^2 |u(t,x)|^2 \, dx \le C\, \big(1+  \int_{\RR^d } x^2 |u_0(x)|^2 \, dx\big) .
$$
\end{enumerate}

\medskip
\noindent
In addition,
\renewcommand{\labelenumi}{(\arabic{enumi})}
\begin{enumerate}
  \item if Hypothesis \ref{hypoas} is satisfied, then we have for all $t\ge 0$,
  $$
   |u(t)|^2 _{L^2 }= |u_0|^2 _{L^2 },\quad \PP-\mbox{a.s.};
$$
  \item if Hypothesis \ref{hypomean} is satisfied, then we have for all $t\ge 0$,
  $$
   \EE |u(t)|^2 _{L^2 }= |u_0|^2 _{L^2 }.
$$

\end{enumerate}

\end{theorem}

\del{\begin{remark}
If  the hypothesis \ref{hypoas} is satisfied, then the solution will a.s. preserves the $L ^2 $ norm.
If  the hypothesis \ref{hypomean} is satisfied, then the solution will preserves the $L ^2 $ norm in mean.
\end{remark}
}

{\begin{example}
Let $\xi \in \RR^d$ be fixed and $\theta :\RR ^d\times \RR\to\RR$. Then
$g_\xi: \RR\ni z\mapsto i(e^{i\theta(\xi,z)}-1)$, and $\hh_\xi  (z)=i\left(\cos(\theta(\xi,z))-1\right)$  satisfy assumption \ref{hypoas}.
\end{example}}

The proof of Theorem \ref{main} is presented in Section \ref{main_proof}, the technical Lemma \ref{local_ex} is presented in Section \ref{proof1}. First, in Section \ref{deterministc} we summarize some deterministic preliminaries
which we  need for the proof.  In the Appendix we collect several results, which we used within the proof.

\section{Deterministic Preliminaries}\label{deterministc}

In this section we shortly introduce some  propositions and lemmata, which are necessary to show our main results.
But before starting let us introduce some definitions.
The group $(\CT(t))_{t\ge0}$, corresponding to the Cauchy problem
\DEQS
 \lk\{ \baray i\partial_t u(t,x)-\Delta u(t,x)&=&0,\\ u(0,x)&=& \phi(x),\quad x\in\RR^d ,\, t\in\RR,
 \earray\rk.
 \EEQS
  can be expressed explicitly in Fourier variables, i.e.\
$$
\CF\lk( \CT(t) \phi\rk)(\xi) = e^ {4\pi^ 2 i|\xi|^ 2 t} \hat \phi(\xi),\quad \xi\in\RR^d ,\, t\in\RR.
$$
We recall  some well known deterministic results.

%

{
\begin{lemma}\label{strichatz}
If $t\not =0$, $1/p+1/p'=1$ and $p'\in[1,2]$, then
$\CT(t):L^{p'}(\RR ^d) \to L^{p}(\RR ^d)$ is continuous and
$$
  \lk| \CT(t) \phi\rk|_{L^{p}}\le  c \,  |t| ^{-\frac d2\lk( \frac 1{p'}-\frac 1p\rk)}  \lk|  \phi\rk|_{L^{p'}}
={ c \, |t| ^{-d\lk( \frac 1{2}-\frac 1p\rk)} }  \lk|
\phi\rk|_{L^{p'}}.
$$
\end{lemma}}

A pair $(p,q)$ is called admissible if
\DEQSZ\label{admissible}
 \lk\{\baray 2\le p<{2d\over d-2},  \;&
\mbox{ if } & d\ge 3,
\\
2\le p<\infty, \; & \mbox{ if } &d=2,
\\2\le p\le \infty, \;& \mbox{ if } &d=1,
\earay\rk\}Ê\quad \mbox{and} \quad \frac 2q=\frac d2 -\frac dp.
\EEQSZ

\medskip

Let $s\in\RR$, $1\le p,q\le \infty$. For an interval $I\subset[0,\infty)$ let $L^q(I;H_p^s(\RR^d))$ be the space of all measurable functions
$f:I\times \RR^d\to \RR$ such that
$$
| f|_{L^q  (I;H^s_p)}:=\lk( \int_I  | f(t)|_{H_p^s}^q\, dt \rk) ^\frac 1q <\infty.
$$
Let us define the convolution operator
\DEQSZ\label{det-con}
\mT u (t) := \int_0^t \CT(t-r) u(r)\, dr, \quad t\ge 0.
\EEQSZ
By means of Lemma \ref{strichatz}, the following Corollary can be proven.
\begin{corollary}\label{caz_simple}(see \cite[Theorem 2.3.3]{cazenave})
\renewcommand{\sr}{s}
Let $(p_0,q_0), (p_1,q_1)\in [2,\infty)\times[2,\infty)$ be two admissible pairs.
Then for all $T>0$ and $s\in\RR$ we have
$$
\lk| \mT u \rk| _{L^{q_0}  (0,T;H^\sr_{p_0})} \le
C \lk|  u \rk| _{L^{q_1'}  (0,T;H^\sr_{p_1'})}.
$$
\end{corollary}


In order to treat the nonlinearity,
let $F:\mathbb{C} \to \mathbb{C}$ be given by $F(u)=|u|^{\alpha-1}u$ and
let $\mF$ be the convolution operator given by
\DEQSZ\label{det-conF}
\mF u (t) := \int_0^t \CT(t-r)F( u(r))\, dr.
\EEQSZ
\begin{remark}\label{f_continuous}
We would like to mention that the Nemitskii operator associated to $F$, defined by
$$
\CF (u) (x) := F(u(x)), \quad u: \RR^d \to \mathbb{C},
$$
is a continuous operator from $L^{\alpha+1} (\RR^d )$ to $L^{\alpha+1\over \alpha}(\RR^d )$ and  hence
from $H^1 (\RR^d )$ to $L^{\alpha+1\over \alpha}(\RR^d )$.
\end{remark}


\begin{proposition}\label{caz_det}
Assume $1<\alpha<1+4/d$.
Let $p=\alpha+1$ and $q=4(\alpha+1)/d(\alpha-1)$.
Then  we have for all $T>0$ and any admissible pairs $(m,l)$,
$$
\lk| \mF u \rk| _{L^{l} (0,T;L^{m})} \le {C }(T) \,  \lk|  u
\rk|^\alpha _{L^{q} (0,T;L^ p)} ,
$$
with $C(T)\to 0$ as $T\to 0$.
\end{proposition}

The proposition can be extended to $H^1(\RR^d)$.
\begin{proposition}\label{caz_det_H1}
Assume
$$
\bcase 1<\alpha<{d+2\over d-2} ,&\;  \mathrm{if}\; d>2;
\\ 1<\alpha<\infty ,&\;  \mathrm{if} \;d=1,2.
\ecase
$$
Let
$$p^\sharp={ d(\alpha+1) \over d+\alpha-1}\quad \mbox{and} \quad q^\sharp={4(\alpha+1)\over (d-2)(\alpha-1)}.
$$
Then we have for all $T>0$,
$$
\lk| \mF u \rk| _{L^{q^\sharp}(0,T;H^ 1 _{p^\sharp} )} \le C\,  \lk|  u
\rk|^\alpha _{L^{q^\sharp} (0,T;H^ 1 _{p^\sharp})} .
$$
\end{proposition}
\begin{proof}
The proof is similar to the proof of Proposition \ref{caz_det}. One only has to take into account the following
estimate, which follows from H\"older's inequality : 
$$
| |u|^ {\alpha-1}\nabla u|_{L^ {{p}^{\sharp'}}} \le c \, | |u|^ {\alpha-1}|_{L^ {l}} |\nabla u|_{L^ {p^\sharp}}\le c \, | u |_{L^ {({\alpha-1})l}}^{\alpha-1} |\nabla u|_{L^ {p^\sharp}}
\le c |\nabla u|_{L^ {p^{\sharp}}}^{\alpha}.
$$
Here $\frac 1l =1-\frac 2{p^{\sharp}}$ and $\frac 1 {(\alpha-1)l}=\frac 1{p^{\sharp}}-\frac 1d$.
Therefore, $(\alpha+1)/p^{\sharp}=(d+\alpha-1)/d$.
\end{proof}

\section{Existence and uniqueness results for
finite \levy measure}\label{proof1}

In this section we show the existence and uniqueness of a solution to \eqref{itoeqn} for a finite \levy measure.
Here, the representation of the \levy process as a finite sum over its jumps is essential.
Using this representation,  the existence and uniqueness of the solution in a pathwise sense can be shown.

\renewcommand{\sn}{{s}}
\begin{tlemma}\label{local_ex}
Let us assume that the \levy measure $\nu$ is finite,  in particular $\nu(Z)=\rho$, and that the Hypothesis \ref{hypogeneral}
is satisfied.
 Let $u_0\in L^2(\Omega;H^ {1}_{2} (\RR^d))$ be fixed and ${\mathcal F}_0$-measurable.
 \medskip

Then, if $\lambda>0$ and
$$1\le \alpha< \bcase  1+4/(d-2)& \mbox{ for } d>2,
\\  + \infty& \mbox{ for }  d=1 \mbox{ or } 2,\ecase
$$
Equation \eqref{itoeqn} has a unique mild  $H^1_{2}(\RR^d)$--valued
solution $u$; in particular, $u$ is $\PP$-a.s.\ \cadlag in  $H^1_{2}(\RR^d)$.
In addition,
for any $T>0$
there exists a constant $C=C(T,C_0(\nu),C_3(\nu),C_g,C_h)$ such that
$$
\EE\sup_{0\le t\le T} |u(t)|_{L^2 }^2 \le C\, \EE |u_0|_{L^2 }^2
$$
and there exists a constant  $C=C(T,C_0(\nu),C_1(\nu),C_g,C_h)>0$ such that
\begin{equation}
\label{energybound}
\EE \sup_{0\le t\le T}  \CH(u(t)) \le  C\, (1+\EE \CH(u_0)).
\end{equation}

\medskip%
\begin{enumerate}
  \item If Hypothesis \ref{hypoas} is satisfied, then  for any $t>0$ we have $\PP$--a.s.
$$
 |u(t)|_{L^2 }^2 = |u_0|_{L^2 }^2.
$$

  \item If Hypothesis \ref{hypomean} is satisfied, then  for any $t>0$ we have
$$
\EE |u(t)|_{L^2 }^2 = \EE |u_0|_{L^2 }^2.
$$
\item If $ \EE  \int_{\RR^d } x^2 |u_0(x)|^2 \, dx <\infty$ and Hypothesis \ref{hypogeneral} (ii)-(c) is satisfied, 
then there exists a constant $C=C(T,C_2(\nu),C_g,C_h)>0$ such that
$$
\EE \sup_{0\le t\le T}  \int_{\RR^d } x^2 |u(t,x)|^2 \, dx \le C\left[1+  \EE  \int_{\RR^d } x^2 |u_0(x)|^2 \, dx +\EE \CH(u_0)\right].
$$

\end{enumerate}
\end{tlemma}
\begin{proof}
%

%
Let $\rho=\nu(Z)$, let $\{ \tau_n:n\in\NN\}$ be a family of independent exponential distributed random variables with parameter $\rho$,
let
\DEQSZ\label{stoppingtimes}
 T_n=\sum_{j=1}^n \tau_j,\quad n\in\NN,
\EEQSZ
and  let $\{ N(t):t\ge 0\}$ be the counting process defined by
$$ N(t) :=\sum_{j=1}^\infty 1_{[T_j,\infty)}(t),\quad t\ge 0.
$$
Observe, for any $t>0$, $N(t)$ is a Poisson distributed random variable with parameter $\rho t$.
Let $\{ Y_n:n\in \NN\}$ be a family of independent, $\nu/\rho$ distributed random variables. Then
the \levy process $L$ given by
$$
L(t)=\int_0^t \int_Z z\,\tilde \eta(dz,ds),\quad t\ge 0,
$$
can be represented as
$$
L(t) = \bcase - z_\nu t & \mbox{ for } N(t)=0,\\
\sum_{j=1}^{N(t)} Y_j -z_\nu t & \mbox{ for } N(t)>0,
\ecase
$$
where $z_\nu= \int_Z G(z)\, \nu(dz)$ (see e.g.\ \cite[Chapter 3]{tankov}).

Now, by a modification of \cite[Theorem 4.4.1]{cazenave}
there exists a solution of the deterministic equation
\DEQSZ\label{itoeqn-finite}
\hspace{3cm}
 \lk\{ \baray \lqq{ i \, \partial_t u(t)  -  \Delta u(t,x) +\lambda |u(t,x)|^{\alpha-1} u(t,x) \, \hspace{2cm}}&&
\\&=& u(t,x) \int_Z \, \hh (z(x)) \mathcal{}\, \nu (dz) - u(t,x) \, z_\nu ,\\
u(0)&=& u_0, \earay \rk.
 \EEQSZ
which is in $C(\RR_+;H^1_2(\RR^d))$,
with $u(T_1)\in H^1 _2(\RR^d)$.
Indeed, setting $\mg(u)=\int_Zu\hh( z)\, \nu(dz)-uz_{\nu}$, it is not difficult to check that for any $u,v \in
L^{\infty}(0,T;H^1_2(\RR^d))$ we have
$$
|\mg(u)|_{L^{\infty}(0,T;H^1_2)}\le C_{\nu} |u|_{L^{\infty}(0,T;H^1_2)}
$$
and
$$
|\mg(u)-\mg(v)|_{L^{\infty}(0,T;L^2)}\le C_{\nu} |u-v|_{L^{\infty}(0,T;L^2)},
$$
where the Lipschitz constant is given by
$$
{C_\nu = \int_Z |z|_{H^{1}_2} \,\nu(dz) +|z_\nu|_{H^1_2}.}
$$
Hence, setting $p=\alpha +1$ and $q=4(\alpha+1)/d(\alpha-1)$ so that $(p,q)$ is an admissible pair, one may use
as in \cite[Theorem 4.4.1]{cazenave} a fixed point in
$$
\begin{array}{rl}
E:= &\{v \in L^{\infty}(0,T;H^1_2(\RR^d))\cap L^q(0,T;H^1_p(\RR^d)), \\ &\; |v|_{L^{\infty}(0,T;H^1_2)}\le M; \;
|v|_{L^q(0,T;H^1_p)}\le M \}
\end{array}
$$
equipped with the distance
$$
d(u,v)=|u-v|_{L^{\infty}(0,T;L^2)}+|u-v|_{L^q(0,T;L^p)},
$$
and a constant $M$ depending on the initial condition (see \cite[p.\ 95,$\uparrow 2$]{cazenave}).
For sufficiently small $T>0$, we  obtain the existence of a unique local solution.
By uniform bounds, this local solution can be globalized.
\del{Finally, by Hypothesis \ref{hypogeneral}-(ii)-(a), the entity $\sup_{0\le t\le T} \int_{\RR ^d} |x| ^2|u(t,x)| ^2\, dx $ can be estimated.
Indeed, one may check that
\begin{eqnarray*}
\lqq{ \frac{d}{dt} \int_{\RR^d} x^2 |u(t,x)|^2 \, dx =2\int_{\RR^d} x\cdot\Im (  u \nabla \bar u) \, dx} && \\
& & {}+\int_Z\int_{\RR^d} x^2|u(t,x)|^2 \Im (h(z(x))\, dx  \, \nu(dz)
\\
&&{}+\int_Z\int_{\RR^d} x^2|u(t,x)|^2 \Im (g(z(x))\, dx  \, \nu(dz).
\end{eqnarray*}
The H\"older inequality and the fact that $\CH(u(t))$ is finite, gives that
$$ \int_{\RR^d} x^2 |u(t,x)|^2 \, dx
$$
is finite.}
%
\medskip

Let us denote the solution by $u_1$. Since at time $T_1$ a jump with size $-iu(T_1) Y_1$ happens,
we put
$u_2^0=u_1(T_1)(1-\Red{i Y_1})$ and consider a second process, starting at time $0$ in point $u_2^0$.
By Hypothesis \ref{hypogeneral}-(i), we know, $u_2^0\in H^1 _2(\RR^d)$.
Hence, again by Theorem \cite[Theorem 4.4.1]{cazenave} and the previous arguments,
there exists a unique global solution of the deterministic equation
\DEQSZ\label{itoeqn-finite2}
\hspace{2cm} \lk\{ \baray \lqq{ i \, \partial_t u(t)  -  \Delta u(t,x) +\lambda |u(t,x)|^{\alpha-1} u(t,x) \, \hspace{2cm}}&&
\\&=& u(t,x) \int_Z \,\hh( z(x)) \mathcal{}\, \nu (dz) - u(t,x) \, z_\nu ,\\
u(0)&=& u_2^ 0 =u_1(T_1)(1-i Y_1). \earay \rk.
 \EEQSZ
Let us denote the solution on $[0,T_2-T_1]$ by $u_2$.
Iterating this step we get a sequence of solutions $\{u_n:n\in \NN\}$. To be more precise,
let us assume that we are given a solution $u_{n-1}$ on the time interval $[0,T_{n-1}-T_{n-2}]$, where
the family of stopping times $\{T_n:n\in\NN\}$ is defined in \eqref{stoppingtimes}.
Let  $  u_{n}^ 0 =u_{n-1}(T_{n-1}-T_{n-2})(1-iY_{n-1})$.
Then, we denote by $u_n$ the solution of the following (deterministic) problem
\DEQSZ\label{itoeqn-finiten}
\hspace{2cm}\lk\{ \baray \lqq{ i \, \partial_t u(t)  -  \Delta u(t,x) +\lambda |u(t,x)|^{\alpha-1} u(t,x) \, \hspace{2cm}}&&
\\&=& u(t,x) \int_Z \,\hh( z(x)) \mathcal{}\, \nu (dz) - u(t,x) \, z_\nu ,\\
u(0)&=& u_{n}^ 0, 
 \earay \rk.
 \EEQSZ
 So, for each $n\in\NN$ we can construct a solution $u_n$ on the time interval $[0,T_{n}-T_{n-1}]$. In the next step we glue these solutions together
 by putting
for $t\in [T_{n-1},T_{n})$, $n\in\NN$
$$
u(t) := u_n(t-T_{n-1}).
$$
Let us observe, that the jumps take place at the end points at each interval and will be taken into account, by taking as initial starting point for the next solution $u_{n}^0$,
the solution $u_{n-1}$ at the end point $T_{n-1}$ plus the jump. In particular,  we put $u_n(0)=u_{n-1}(T_{n-1}-T_{n-2}) -i u_{n-1}(T_{n-1}-T_{n-2})Y_{n-1}$.
It is straightforward to show, that $u$ solves \eqref{itoeqn}.
Since $\PP\lk( N(T)<\infty\rk)=1$, the solution $u$ is a.s.\ defined on $[0,T]$.
The \cadlag property follows by the fact, that  $\lim_{t\downarrow T_j} u(t)=u_j^0$ and the limit $\lim_{t\uparrow T_j} u(t)$ exists
in $H^1_2(\RR^d)$.

Summing up, we have shown the existence of a unique solution $u$ belonging $\PP$--a.s.\ to $\DD(0,T;H^ 1 _2(\RR^ d ))\cap L^ q(0,T;H^ 1_p(\RR^ d ))$.

\medskip


Next, we will show that under the hypothesis of the Lemma, the mass may be estimated, i.e.\ for $0\le t\le T$,
$$
\EE\, \sup_{0\le t\le T} \CE(u(t))=\EE \sup_{0\le t\le T} |u(t)|_{L^2 }^2 \le C\EE |u_0|_{L^2 }^2=C\EE \,\CE(u_0),
$$
where $C=C(T,C_0(\nu),C_3(\nu),C_g,C_h)$.
In a first step we are aiming to prove
\DEQSZ\label{toshowl22}
\EE\sup_{0\le s \le t} \lk[ |u (s)|_{L ^2} ^2 -|u(0)|^ 2 _{L^ 2 }\rk]
&\le &  C\EE \lk( \int_0^ t |u(s)|^ 4 _{L^ 2 }  \, ds\rk) ^ \frac 12
,
\EEQSZ
where $C=C(C_0(\nu),C_3(\nu),C_g,C_h)$.
Assume for the time being that \eqref{toshowl22} is true. Then, it follows by the H\"older inequality
\DEQS
 \EE \sup_{0\le s\le t}\lk[ |u (s)|_{L ^2} ^2 -|u(0)|^ 2 _{L^ 2 }\rk]
&\le &  C \sqrt{t} \EE  \sup_{0\le s\le t}  |u (s)|_{L ^2} ^2. 
\EEQS
If $t^ \ast$ is small enough that  $C\,\sqrt{t^\ast}\le \frac 12 $, then 
$$
\EE \sup_{0\le s\le t^ \ast} \lk[ |u (s)|_{L ^2} ^2\rk]\le 2 \EE |u(0)|^ 2 _{L^ 2 }.
$$
Iterating this step 
 we get
$$
\EE \sup_{0\le s\le T}   |u (s)|_{L ^2} ^2\le C\, \EE |u(0)|^ 2 _{L^ 2 },
$$
where $C=C(T,C_0(\nu),C_3(\nu),C_g,C_h)$.

\medskip

Let us show estimate \eqref{toshowl22}. If we denote by $u ^c$ the continuous part of $u$ and put $f(z)=(-iG(z))$, we get by the It\^o formula for a twice Frechet differentiable function $\Phi:L ^2(\RR ^d)\to\RR$
\DEQS
\lqq{
d\Phi (u(t))= \Phi '(u(t)) du ^c (t)}
&&\\
&& +\int_Z\left[ \Phi (u(t^-)(1+f(z)))-\Phi (u(t^-)\right]\tilde  \eta(dz,dt)
\\
&&{}+ \int_Z\left[ \Phi (u(t^-)(1+f(z)))-\Phi (u(t^-)- \Phi '(u(t))\,[u(t ^-)\,  f(z)]\right]\gamma(dz,dt).
\EEQS
First, note, since on each interval the solution belongs to $H^ 1 _2(\RR^d)$,
all terms in the above  It\^ o formula are well defined.
Additionally, with
 \DEQSZ\label{f-dev}
\lqq{ {du (t^-,x) } = \lk(-i  \Delta  u (t^-,x) +i\lambda|u (t^-,x)|^{\alpha-1}
u (t^-,x) \rk) \, dt}
&&
\\ && {} -i \int_Z u (t^-,x)g(z(x))\, \tilde \eta(dz,dt)- i\int_Z u (t^-,x)\, \hh(z(x)) \,\gamma(dz, dt),\quad x\in\RR^d,\, t\ge 0 \nonumber,
\EEQSZ
one obtains 
\DEQSZ\label{goingback1}
\lqq{d|u (t)|_{L ^2} ^2
= -  \int_Z\int _{\RR ^d }
2   u (t,x) \overline{ u (t,x)} \Im(\overline{\hh (z(x)})\, dx \, \nu(dz) \, dt }
\\
\nonumber & & + \int_Z\int _{\RR ^d } u(t^-,x)\overline{ u(t^-,x)} \lk[\lk| 1-i g(z(x))\rk|^2-1  \right]\, \, dx\,\tilde \eta(dz, dt)
\\
\nonumber & & + \int_Z\int _{\RR ^d } u(t,x)\overline{ u(t,x)} \lk[\lk| 1-i g(z(x))\rk|^2-1 +2\Im \left( \overline{ g(z(x))} \right) \right]\, \, dx\,\gamma(dz, dt)
\\ 
\nonumber
&= & 2 \int_Z
\int _{\RR ^d }   u (t,x) \overline{ u (t,x)} \Im({\hh (z(x)}) \, dx\, \nu(dz) \, dt
\\
\nonumber & & + \int_Z\int _{\RR ^d } u(t^-,x)\overline{u(t^-,x)} \lk[\lk| 1-i g(z(x))\rk|^2-1  \right]\, \, dx\,\tilde \eta(dz, dt)
\\
\nonumber & & + \int_Z\int _{\RR ^d } u(t,x)\overline{u(t,x)} \lk[\lk| 1-i g(z(x))\rk|^2-1 - 2\Im \left(  g(z(x)) \right) \right]\, \, dx\,\gamma(dz, dt).
\EEQSZ

To be more precise, one has by direct calculations 
\DEQSZ\label{goingback11}
\hspace{2cm}
d|u (t)|_{L ^2} ^2
&=&  2 \int_Z
\int _{\RR ^d }   u (t,x) \overline{ u (t,x)} \Im({\hh (z(x)}) \, dx\, \nu(dz) \, dt
\\\nonumber
& &{} + \int_Z\int _{\RR ^d } u(t^-,x)\overline{u(t^-,x)} \lk[\lk| g(z(x))\rk|^2 + 2 \Im(g(z(x)))  \right]\, \, dx\,\tilde \eta(dz, dt)
\\
\nonumber & & {}+ \int_Z\int _{\RR ^d } u(t,x)\overline{u(t,x)} \lk| g(z(x))\rk|^2\, \, dx\,\gamma(dz, dt).
\EEQSZ
\noindent

An application of the Burkholder inequality and Minkowski inequality yields
\DEQSZ
\nonumber
\lqq{\EE\sup_{0\le s \le t} \lk[ |u (s)|_{L ^2} ^2 -|u(0)|^ 2 _{L^ 2 }\rk]
}
\\  \nonumber
&\le& \EE \lk(\int_0^ t \int_Z \lk( \int _{\RR ^d } u(t,x)\overline{u(t,x)} \lk[\lk| g(z(x))\rk|^2 + 2 \Im(g(z(x)))  \right]\, \, dx\rk)^ 2
 \nu (dz)\, ds\rk) ^ \frac 12
\\
\nonumber
&& {}+ \EE \int_0^ t  \lk| \int_Z
2 \int _{\RR ^d }   u (t,x) \overline{ u (t,x)} \Im({\hh (z(x)}) \, dx \nu(dz)\rk| \, ds
\\
\nonumber
&&
{}+  \EE \int_0^t \int_Z\int _{\RR ^d } u(s,x)\overline{ u(s,x)} \lk| g(z(x))\rk|^2\, \, dx\,\gamma(dz, dt)
.
\EEQSZ
Taking into account Hypothesis \ref{hypogeneral}, we know that there exists a constant $C=C(T,C_0,C_3,C_g,C_h)>0$ such that
\DEQS
\nonumber
\EE\sup_{0\le s \le t} \lk[ |u (s)|_{L ^2} ^2 -|u(0)|^ 2 _{L^ 2 }\rk]
&\le & C \EE \lk( \int_0^ t |u(s)|^ 4 _{L^ 2 }  \, ds\rk) ^ \frac 12,
\EEQS
that is \eqref{toshowl22} holds and the estimate on the mass follows as explained above.
\medskip

If Hypothesis \ref{hypoas} is satisfied, one easily deduces from \eqref{goingback1} that
 $ \CE(u(t))=  \CE(u(0))$ for all $t\in[0,T]$.
If only Hypothesis \ref{hypomean} is satisfied,
then one easily deduces from \eqref{goingback11} that
$\EE |u (t)|_{L ^2} ^2=\EE |u (0)|_{L ^2} ^2$ for all $t\in[0,T]$.

\noindent

\medskip


In a second step
we will prove that there exists a constant $C>0$ such that
\DEQSZ\label{hamil}
\EE \sup_{t\in [0,T]} \CH(u(t))   &\le& C\,\lk( 1+ \EE \CH (u(0))\rk).
\EEQSZ

\medskip

\del{On the other hand, one may compute, using \eqref{hamiltonian} and \eqref{itoeqn-finite} :
\DEQS
\lqq{ \frac{d}{dt} \CH(u(t)) }
\\
&=& \frac12\int_Z \int_{\RR^d} \Re \lk( \nabla \bar u \cdot  \overline{ \nabla \left(-i u \hh(z(x))\right)}\rk)\,dx\, \nu(dz) \,
\\
&&{}+\frac12\int_Z \int_{\RR^d} |u|^{\alpha+1} \Re\lk( \overline{-i g(z(x))}\rk) \,dx\, \nu(dz)
\\
&=& \frac12\int_Z \int_{\RR^d} \Re \lk( \nabla \bar u \cdot  \overline{ \nabla \left(-i u \hh(z(x))\right)}\rk)\,dx\, \nu(dz) \,
\\
&&{}+\frac12\int_Z \int_{\RR^d} |u|^{\alpha+1} \Im\lk(g(z(x))\rk) \,dx\, \nu(dz).
\EEQS
Applying Hypothesis \ref{hypogeneral} and using the sequence of estimates
$$
|\la \nabla u, \nabla(u\hh(z))\ra| \le C |u|^2 _{H^1_2}\left( |\CCC(z)|_{L^{\infty}}^2+|\nabla \CCC(z)|_{L^{\infty}}\right),
$$
 one can show that for any $T>0$,
$$\CH(u(t))\le C(T) \, C_1(\nu)(C_g+C_h),\quad \forall t\in[0,T]. $$
By these uniform bounds, the existence and uniqueness of the global solution to \eqref{itoeqn-finite} can be shown.
}

In order to justify the computation of  the It\^o formula for the Hamiltonian $\CH$, one needs also to regularize the Hamiltonian. In particular, one needs to regularize both terms in the Hamiltonian. One possibility is to define $J^\ep:=(I-\ep\Delta)^{-1}$ and to consider
$$
\CH_\ep (u) =\CH (J_\ep^\frac12 u) = \frac 12 \int_{\RR ^d}| J_\ep ^\frac 12 \nabla u| ^2 +\frac{\lambda}{\alpha+1} \int _{\RR^ d}|J ^\frac 12 _\ep u|^{\alpha+1} \, dx.
$$
Note, since $u$ belongs $\PP$-a.s.\, to $H_2^ 1(\RR^d)\subset L^{\alpha+1}(\RR^d)$  for any $t\in[0,T]$, we know $\PP$--a.s.\  $\CH_\ep(u(t))\to \CH(u(t))$ for $\ep\to 0$.
Taking into account that $\PP$--a.s.\ $u$ belongs to $\DD([0,T];H^ 1 _2(\RR^d))$, in addition, by Theorem 7.8-(b) \cite{ethier}, it follows that the process $[0,T]\ni t \mapsto \CH_\ep(u(t))$ converges to the process
$[0,T]\ni t \mapsto \CH(u(t))$ in $\DD(0,T;\RR)$.
\del{Let us also note that 
\DEQS
\lqq{ \CH_\ep'(u)\cdot v= -\la J_\ep \Delta u,v\ra }
&&
\\ && {} +\lambda \la J_\ep ( |u| ^{\alpha-1}u),v\ra
+\lambda \la J_\ep(|u| ^{\alpha-1}v,u\ra +(\alpha-1) \lambda \la J_\ep (|u| ^{\alpha-3}u\Re(u\bar u),u\ra.
\EEQS
}
\medskip

Let us apply the It\^o formula to $\CH_\ep(u(t))$.
First, note that
$$\CH_\ep'(u)\cdot v= \CH'(J_\ep^\frac12 u)\cdot J_\ep^\frac12 v
=  \la -J_\ep^\frac12 \Delta u,J_\ep^\frac12 v\ra +\lambda \la |J_\ep^\frac12 u|^{\alpha-1}J_\ep^\frac12 u,J_\ep^\frac12 v\ra.
$$
Using 
 \DEQSZ\label{f-dev-h}
\lqq{ {du (t,x) } = \lk(-i  \Delta  u (t,x) +i\lambda |u (t,x)|^{\alpha-1}
u (t,x) \rk) \, dt}
&&
\\ && {} -i \int_Z u (t^-,x)g(z(x))\, \tilde \eta(dz,dt)- i\int_Z u (t,x)\, \hh (z(x)) \,\gamma(dz, dt),\quad x\in\RR^d,\, t\ge 0, \nonumber
\EEQSZ
 \DEQSZ\label{f-devc-h}
\lqq{ {du ^c (t,x) } = \lk(-i  \Delta  u (t,x) +i\lambda |u (t,x)|^{\alpha-1}
u (t,x) \rk) \, dt}
&&
\\ && {} - i\int_Z u (t,x)\, \hh (z(x)) \,\gamma(dz, dt),\quad x\in\RR^d,\, t\ge 0, \nonumber
\EEQSZ
and the It\^ o formula
\DEQS
\lqq{
d\CH_\ep (u)= \CH_\ep'(u(t)) du ^c (t)}
&&\\
&& +
\int_Z \CH_\ep' (u(t^-))\lk[ u(t^-)(-iG(z))\rk]\tilde  \eta(dz,dt)
\\
&&{}+ \int_Z\left[ \CH_\ep (u(t^-)(1-iG(z)))-\CH_\ep (u(t^-))- \CH_\ep '(u(t))\,[(-i) u(t ^-)\,  G(z)]\right]\eta(dz,dt),
\EEQS
%
 we get
\DEQSZ\label{hierstart}
\lqq{ \CH_\ep(u^\ep(t))= \CH_\ep(u^\ep(0))-\int_0^t \la \CH_\ep'(u(s)), i  \lk(\Delta u(s)-\lambda |u(s)|^{\alpha-1}u(s)\rk)\ra ds}
&&
\\\nonumber
&&{}
-\int_0^t\int_Z \la \CH_\ep'(u(s)),i u(s)H(z)\ra \nu(dz)\, ds
\\\nonumber
&&{}
+\int_0^t \int_Z\lk[ \CH_\ep(u(s-)(1-iG(z))) -\CH_\ep(u(s-))\rk]\tilde \eta(dz, ds)
\\\nonumber
&&{}
+\int_0^t \int_Z \lk[ \CH_\ep(u(s-)(1-iG(z))) -\CH_\ep(u(s-))\rk.
\\\nonumber
&& {}\phantom{\Bigg|} \quad \lk.- \CH_\ep'(u(s-))( u(s-)(-iG(z)))\rk]\gamma(dz, ds)
.
\EEQSZ
In order to analyse the first term, we first use  the fact that $\la J_\ep^\frac12 \CH'(u),iJ_\ep^\frac12 \CH'(u)\ra =0$
to write
\begin{eqnarray*}
\CH'_\ep(u)\cdot [i(\Delta u-\lambda |u|^{\alpha-1}u)] & = & -\la\CH'(J_\ep^\frac12 u),iJ_\ep^\frac12 \CH'(u)\ra\\
& = &-\la \CH'(J_\ep^\frac12 u)-J_\ep^\frac12 \CH'(u),iJ_\ep^\frac12 \CH'(u)\ra.
\end{eqnarray*}
Note that all the terms are well defined, since $u\in \DD([0,T];H^ 1 _2(\RR^d))$ implies $\CH'(u)\in \DD([0,T]; H^{-1}_2(\RR^d))$
and $J_\ep^\frac12 \CH'(u)$ is a.s. in $L^{\infty}(0,T;L^2(\RR^d))$, while $J_\ep^\frac12 u \in L^{\infty}(0,T;H^2_2(\RR^d)$ so that
$\CH'(J_\ep^\frac12 u)\in L^{\infty}(0,T;L^2(\RR^d))$.

Next, we prove that $\int_0^t\la \CH'(J_\ep^\frac12 u)-\Jeh \CH'(u),i\Jeh \CH'(u)\ra \,ds$ tends to zero as $\ep$ tends to zero, for any
$t$, a.s. First, note that
\begin{equation}
\label{comu}
\CH'(\Jeh u)-\Jeh \CH'(u)=\lambda |\Jeh u|^{\alpha-1}\Jeh u - \lambda \Jeh(|u|^{\alpha-1}u).
\end{equation}
Let us remind that the solution belongs $\PP$-a.s.\ to $\DD(0,T;H^ 1 _2)\cap L^ q (0,T;H_ p^ 1(\RR^d))$.
Then, using H\"older inequalities and Sobolev embeddings, it is easily seen that the above term is bounded independently
of $\ep$ in $H_{\alpha+1/\alpha}^1(\RR^d)\subset L^2(\RR^d)$; indeed, one may e.g. bound
$$
\left| |\Jeh u|^{\alpha-1}\nabla \Jeh u\right|_{L^{\frac{\alpha+1}{\alpha}}} \le C|\Jeh u|_{L^{\alpha+1}}^{\alpha-1} |\nabla \Jeh u|_{L^{\alpha+1}}
\le C |u|_{H^1_2}^{\alpha-1}Ê|\Jeh \nabla u|_{L^{\alpha+1}} \le C |u|_{H^1_2}^{\alpha-1} |u|_{H^1_{\alpha+1}},
$$
since $\Jeh$ is a bounded operator -- with a bound independent of $\ep$ -- in $L^{\alpha+1}(\RR^d)$. All the other
terms are estimated in the same way. Hence, we know by \eqref{comu} that
\begin{equation}
\label{comuep}
| \CH'(\Jeh u)-\Jeh \CH'(u)|_{H^1_{\frac{\alpha+1}{\alpha}}}\le C |u|_{H^1_2}^{\alpha-1}|u|_{H^1_{\alpha+1}}.
\end{equation}
With the same arguments,
\begin{equation}
\label{nonl}
\left| |u|^{\alpha-1}u\right|_{H^1_{\frac{\alpha+1}{\alpha}}}\le C |u|_{H^1_2}^{\alpha-1}|u|_{H^1_{\alpha+1}}.
\end{equation}
Let us now decompose
\begin{eqnarray*}
& & \la \CH'(\Jeh u)-\Jeh \CH'(u), i\Jeh \CH'(u)\ra \\
&=& -\la \CH'(\Jeh u)-\Jeh \CH'(u), i\Jeh \Delta u\ra +\lambda \la \CH'(\Jeh u)-\Jeh \CH'(u), i\Jeh (|u|^{\alpha-1}u)\ra
\end{eqnarray*}
and integrate the first term by parts. We then have to consider
\begin{eqnarray*}
I_\ep&=& \int_0^t \la \nabla(\CH'(\Jeh u)-\Jeh \CH'(u)),i\Jeh \nabla u\ra\, ds\\
& =& \int_0^t \la \Jeh \nabla (\CH'(\Jeh u)-\Jeh \CH'(u)), i\nabla u\ra \, ds.
\end{eqnarray*}
Using \eqref{comuep}, $\Jeh \nabla (\CH'(\Jeh u)-\Jeh \CH'(u))$ is bounded in $L^2(0,T;L^{\frac{\alpha+1}{\alpha}}(\RR^d))$,
independently of $\ep$ by
$$C|u|_{L^{\infty}(0,T;H^1_2)}^{\alpha-1} |u|_{L^2(0,T;H^1_{\alpha+1})}\le C |u|_{L^{\infty}(0,T;H^1_2)}^{\alpha-1}
|u|_{L^q(0,T; H^1_{\alpha+1})}.$$
Hence, $\Jeh \nabla (\CH'(\Jeh u)-\Jeh \CH'(u))$ converges weakly to $0$ in $L^2(0,T;L^{\alpha+1/\alpha}(\RR^d))$.
Since $\nabla u\in L^2(0,T;L^{\alpha+1}(\RR^d))$, it follows that $I_\ep$ converges to $0$ as $\ep \to 0$.
For the second term, we use the same argument and \eqref{nonl} : by \eqref{comuep} and the embedding $H^1_{\alpha+1/\alpha}(\RR^d)
\subset L^2(\RR^d)$, the term $\Jeh (\CH'(\Jeh u)-\Jeh \CH'(u))$ is bounded uniformly in $\ep$, in $L^2(0,T;L^2(\RR^d))$,
by $C|u|_{L^{\infty}(0,T;H^1_2)}^{\alpha-1} |u|_{L^2(0,T;H^1_{\alpha+1})}$, and $|u|^{\alpha-1}u \in L^2(0,T;H^1_{\alpha+1/\alpha}(\RR^d))
\subset L^2(0,T;L^2(\RR^d))$ by \eqref{nonl}, so that again,
$$
 \int_0^t \la \CH'(\Jeh u)-\Jeh \CH'(u),i\Jeh (|u|^{\alpha-1} u)\ra\, ds
= \int_0^t \la \Jeh (\CH'(\Jeh u)-\Jeh \CH'(u)), i|u|^{\alpha-1}u \ra \, ds
$$
converges to $0$ as $\ep \to 0$.

Going back to \eqref{hierstart}, the first term which does not vanish for $\ep\to 0$ is
\DEQS
\int_0^t\int_Z \la \CH_\ep'(u(s)),i u(s)H(z)\ra \nu(dz) \,ds.
\EEQS
However, taking $\ep\to 0$ we see by similar arguments as before that
\DEQS
\int_0^t\int_Z \la \CH_\ep'(u(s)),i u(s)H(z)\ra \nu(dz)\,ds\to \int_0^t\int_Z \la \CH'(u(s)),i u(s)H(z)\ra \nu(dz)\,ds.
\EEQS

\del{\\\nonumber
&&{}
+\int_0^t \int_Z\lk[ \CH_\ep(u(s-)(1-G(z))) -\CH_\ep(u(s-))\rk]\tilde \eta(dz, ds)
\\\nonumber
&&{}
+\int_0^t \int_Z \lk[ \CH_\ep(u(s-)(1-G(z))) -\CH_\ep(u(s-))\rk.
\\\nonumber
&& {}\phantom{\Bigg|} \quad \lk.-\la \CH_\ep'(u(s-))( u(s-)(-iH(z)))\rk]\nu(dz, ds)
.}
\noindent
Straightforward calculations give
\DEQS
\lqq{ \int_0^t \, \CH'(u(t)).(- i\int_Z u (t,x)\, \CCC (z)  \,\nu(dz))\,dt}
\\ &
=&-\int_0^t  \langle \Delta u(t), - i\int_Z u (t)\, \CCC(z)  \,\nu(dz)\rangle\,dt
\\
&&{}+\lambda \int_0^t \langle |u(t)|^{\alpha -1} u(t)- i\int_Z u (t,x)\, \hh (z(x))
\,\nu(dz)\rangle\,dt.
\EEQS
Applying integration by parts we get for the first summand 
\DEQSZ\label{herecompare}
\lqq{-
\int_Z\la \Delta u (t), -iu (t)\CCC(z)\ra\,\nu(dz)}
\\\nonumber
& =&
\int_Z\int_{\RR^d}\Re( \nabla u(t,x) \nabla( \overline{ -i  u(t,x) \hh (z(x))   } ) \, dx \,\nu(dz)
\\\nonumber&=&-
\int_Z\int_{\RR^d}\Im( \nabla u(t,x) \nabla(  \overline{   u(t,x) \hh (z(x))   } ) \, dx \,\nu(dz)
\\
\nonumber  &=&
  \int_Z\int_{\RR^d}|\nabla u(t,x)| ^2 \Im(  {  \hh (z(x))} )\, dx \,\nu(dz)
\\ \nonumber && \hspace{2cm} -
 \int_Z\int_{\RR^d}\Im( \nabla  u(t,x) \overline{ u(t,x)} \nabla \overline{  \hh (z(x))} )\,dx \,\nu(dz)
 \EEQSZ
Next,
\DEQSZ\label{herecomparealpha}
\lqq{\lambda
\int_Z\la | u| ^{\alpha-1} u ,-iu H(z)\ra\,\nu(dz)}
&&
\\ \nonumber & =&\lambda
\int_Z\int_{\RR^d}\Re\lk( | u(t,x) | ^{\alpha-1} u(t,x)  (\overline{ -i  u(t,x) \hh (z(x)) )  } \rk) \, dx \,\nu(dz)
\\\nonumber
&=&
- \lambda \int_Z\int_{\RR^d}\Im\lk(  | u(t,x) | ^{\alpha-1} u(t,x)  \overline{   u(t,x)} \overline{\hh (z(x))   } \rk) \, dx \,\nu(dz)
\\\nonumber
&=& \lambda
 \int_Z\int_{\RR^d}| u(t,x) | ^{\alpha+1}\Im\lk(  \hh (z(x)) \rk) \, dx \,\nu(dz).
\EEQSZ
In the next lines we calculate the terms arising due to the jumps, that is  the terms
\DEQSZ\label{ft}
\int_Z \CH_\ep' (u(t^-))\lk[ u(t^-)(-iG(z))\rk]\tilde  \eta(dz,dt)
\EEQSZ
and
\DEQSZ\label{ST}
\int_Z\left[ \CH_\ep (u(t^-)(1-iG(z)))-\CH_\ep (u(t^-))- \CH_\ep '(u(t))\,[-iu(t ^-)\,  G(z)]\right]\eta(dz,dt).
\EEQSZ
Here, again one has to take the limit. Since, the arguments are similar as before we omit them.
Applying the  Burkholder inequality we get
\DEQS
\lqq{
\EE\sup_{0\le s\le t} \lk| \int_Z\CH' (u(s^-))\lk[ u(s^-)(-iG(z))\rk]\tilde  \eta(dz,ds)\rk|  }
\\
&\le &
\EE \lk( \int_0 ^t\int_Z  \lk| \CH' (u(s^-))\lk[ u(s^-)(-iG(z))\rk]\rk| ^2\nu(dz) \, ds\rk) ^\frac 12
\\
&\le&2
\EE\lk( \int_0 ^t\int_Z \lk| \la \Delta u (s), (-iu (s)G(z))\ra \rk| ^2\nu(dz) \, ds\rk.
\\
&&\lk.
{}+ \int_0 ^t \int_Z\lk|  \la | u(s ^-)|^ {\alpha-1}  u(s ^-), u(s ^-) (-iG(z)) \ra\rk|^ 2 \nu(dz) \, ds
\rk) ^\frac 12
.
\EEQS
In order to  calculate the inner part of the first term we compare it to \eqref{herecompare} and get
\DEQS
\lqq{\EE\lk(\int_0 ^t \int_Z \lk| \la \Delta u (s), -iu (s)G(z)\ra \rk| ^2\nu(dz) \, ds\rk)^ \frac 12 }
\\
&\le &
C\EE\lk( \int_0 ^t \int_Z\lk( \int_{\RR^d}|\nabla u(s,x)| ^2 |\Im(  {  g(z(x))} )| \, dx\rk) ^ 2 \,\nu(dz)\,ds\rk.
\\ \nonumber && \hspace{2cm} +\lk.
\int_0 ^t \int_Z\lk( \int_{\RR^d}|\Im( \nabla  u(s,x) \overline{ u(s,x)} \nabla \overline{  g(z(x))} )| \,dx\rk) ^ 2  \,\nu(dz)\rk) ^ \frac 12
.
\EEQS
In order to  calculate the inner part of the second term we compare it to \eqref{herecomparealpha} and get
\DEQS
\lqq{\EE\lk(\int_0 ^t \int_Z\lk|  \la | u(s ^-)|^ {\alpha-1}  u(s ^-), u(s ^-) (-iG(z)) \ra\rk|^ 2 \nu(dz) \, ds\rk)^ \frac 12 }
&&
\\
&\le &
\EE\lk(\int_0 ^t  \int_Z\lk( \int_{\RR^d}| u(s,x) | ^{\alpha+1}|\Im\lk(   {  g(z(x))  } \rk)|  \, dx\rk) ^ 2 \,\nu(dz)\, ds\rk)^ \frac 12
.
\EEQS

Now we are going to calculate the term \eqref{ST}.
Here, taking into account that
$$\int_Z\left|\CH (v(1-iG(z)))-\CH (v)- \CH'(v)\,[-i v\,  G(z)]\right| \, \nu(dz)<\infty
$$
for all $v\in H^ 1 _2(\RR^d)$,
taking the expectation leads on both sides, to 
\DEQS
\lqq{ \EE \int_0^ t \int_Z\left[ \CH (u(s^-)(1-iG(z)))-\CH (u(s^-))- \CH '(u(s^-))\,[-i u(s ^-)\,  G(z)]\right]\eta(dz,ds)
}&&
\\ &\le &\EE \int_0^ t \int_Z\left| \CH(u(s^-)(1-iG(z)))-\CH (u(s^-))- \CH '(u(s^-))\,[-i u(s ^-)\,  G(z)]\right|\,\nu(dz)\, ds
.
\EEQS
First we will calculate the terms of $\CH$ involving $ \int_{\RR ^d}|\nabla u(x)| ^2 \, dx$.
Here, we will  use the identity  $|a| ^2 -|b| ^2 =\Re((a-b)\overline{(a+b)} )$, and taking expectation gives

\DEQS
\lqq{
\frac 12 \EE\int_0^ t \int_Z\lk[ \lk| \nabla\lk( u(s^-) (1-iG(z)) \rk)  \rk| ^2\rk.}
&&\\
&&{}\lk. -\lk|\nabla\lk( u(s^-) \rk)\rk| ^2  -2 \la \nabla u(s^-) ,\nabla ( u(s^-)(-i G(z) ))\ra
\rk]\,
\nu(dz)\,ds
\\
&=&
\frac 12 \EE\int_0^ t \int_Z \int_{\RR^ d }\lk[\Re\lk\{\lk( \nabla( u(s^-,x)(-ig(z(x)))) \rk)\overline{ \lk( \nabla( u(s^-,x)(2-ig(z(x))) \rk)} \rk\}
\rk.
\\
&&{} -
\lk.
2     \Re \lk\{ \nabla u(s^-,x) \cdot \overline{ \nabla( u(s^-,x) (-ig(z(x)))}\rk\}\,
\rk]\, dx\,
\nu(dz)\, ds
\\
&=&
\frac 12 \EE\int_0^ t \int_Z \int_{\RR^ d }\lk[\lk|\nabla \lk(-i u(s^-,x) g(z(x))) \rk) \rk|^2
+2\Re \lk(  \overline{ \nabla u(s^-,x)}\nabla (-iu(s^-,x)g(z(x)))  \rk) \rk\}
\\
&&{} -
\lk.
2     \Re \lk\{ \nabla u(s^-,x) \cdot \overline{ \nabla( u(s^-,x) (-ig(z(x))))}\rk\}\,
\rk]\, dx\,
\nu(dz)\, ds
\\
&=&
\frac 12 \EE\int_0^ t \int_Z \int_{\RR^ d }\lk|\nabla \lk(-i u(s^-,x) g(z(x)) \rk) \rk|^2
\,dx\,
\nu(dz)\, ds.
\EEQS
It remains to calculate the second part of \eqref{ST}, i.e.\
\DEQS
\lqq{ \frac{\lambda}{\alpha +1} \EE  \int_0^ t \int_Z  \int _{\RR ^d } \lk[ |u(s^-,x)(1-ig(x))| ^{\alpha +1} \phantom{\Big|} \rk.
}
&&
\\&& {}\lk.
-|u(s^-,x)| ^{\alpha +1}-(\alpha+1) |u(s^-,x)|^ {\alpha-1}\Re\lk( u(s^-,x) \overline { u(s^-,x)(-i g(z(x))}\rk)\rk] dx\,\nu(dz)\, ds.
\EEQS
{The Taylor formula yields}
\del{or, since $\Re i a=-\Im a$,
\DEQS
\lqq{\frac{1}{\alpha +1} \EE \int_0^ t \int_Z  \int _{\RR ^d }  \lk[ |u(s^-,x)(1-ig(x))| ^{\alpha +1} \phantom{\Big|}\rk.
}
&&
\\ &&{}\lk.
-|u(s^-,x)| ^{\alpha +1}-(\alpha+1) |u(s^-,x)|^ {\alpha+1}\Im\lk( \overline{  g(z(x))}\rk)\rk]  dx\,\nu(dz)\, ds.
\EEQS
This term can be estimated by an application of the Taylor formula. In particular, we get
}
\DEQS
\ldots &\le &
C  \int_0^ t \EE \int_Z  \int _{\RR ^d } |u(s^-,x)| ^{\alpha +1} |g(z(x))|^ 2 dx\,\nu(dz)\, ds
\EEQS
Collecting altogether, taking into account the Hypothesis \ref{hypogeneral}, and rearranging the terms we see that
the bound \eqref{toshowh} below is satisfied ; indeed, first, observe, since no stochastic integral is involved in the bounds,
we can change from $s^-$ to $s$, and we have

\DEQS
\nonumber
\lqq{\EE \sup_{0\le s\le t} \lk[ \CH(u(s))-\CH(u(0))\rk]}
\\
 &\le&\phantom{\Bigg|}
  C\, \EE \int_0^ t\int_Z\int_{\RR^d}\lk|\nabla u(s,x)\rk| ^2 |\Im(  {  \hh (z(x))} )|\, dx \,\nu(dz)\, ds
\phantom{\Bigg|}
 +C\, \EE \int_0^ t\int_Z\int_{\RR^d}\lk\{ {\lk| \Im( \nabla  u(s,x) \overline{ u(s,x)} \nabla \overline{  \hh (z(x))} )\rk|} \rk.
\phantom{\Bigg|}
\\
&&{} \lk.
+ |u(s,x)|^{\alpha+1} |\Im(h(z(x)))| + \lk| \nabla \lk(-iu(s,x)g(z(x))\rk)\rk|^2 + |u(s,x)|^{\alpha+1}|g(z(x))|^2 \rk\}
\,dx\,
\nu(dz)\, ds
\phantom{\Bigg|}
\\
&&{}+
\EE\lk( \int_0^ t \int_Z\lk( \int_{\RR^d}|\nabla u(s,x)| ^2 |\Im(  {  g (z(x))} )| \, dx\rk) ^ 2  \,\nu(dz)\, ds\rk)^ \frac 12
\\ \nonumber && +
\EE\lk( \int_0^ t  \int_Z\lk( \int_{\RR^d}|\Im( \nabla  u(s,x) \overline{ u(s,x)} \nabla \overline{  g (z(x))} )| \,dx\rk) ^ 2  \,\nu(dz)\, ds\rk) ^ \frac 12
\\
&&{}+
\EE\lk(  \int_Z\lk( \int_{\RR^d}\lk| u(s,x) \rk| ^{\alpha+1}\lk|  g (z(x) ) \rk|  \, dx \rk) ^ 2 \,\nu(dz)\, ds\rk)^ \frac 12
.
\EEQS
 {Next,
carefully applying the H\"older inequality and, if necessary, the Young inequality term by term, and using
Hypothesis \ref{hypogeneral}-(i)
one finally arrives at}
\DEQSZ
\label{toshowh}
\lqq{
\EE \sup_{0\le s\le t} \lk[ \CH(u(s))-\CH(u(0))\rk]
} \nonumber
&&\\
&\le& C\lk( 1 + \EE \lk( \int_0 ^t  \CH(u(s)) ^2 \, ds \rk) ^\frac 12 +\EE  \int_0 ^t  \CH(u(s))\, ds \rk) ,
\EEQSZ
where the constant $C>0$ depends only on $C_0(\nu)$, $C_1(\nu)$, $C_2(\nu)$, and $\alpha$.
%
%
We deduce

\DEQS
\lqq{
\EE \sup_{0\le s\le t} \lk[ \CH  (u(s))-\CH  (u(0))\rk]
}
&&\\
&\le&  C + C \sqrt{t}\, \EE  \sup_{0\le s\le t} \CH  (u(s)) +C \EE  \int_0 ^t  \CH  (u(s))\, ds
\\
&\le&  C + C\lk(  \sqrt{t}  + t\rk)\, \EE  \sup_{0\le s\le t} \CH  (s)  .
\EEQS
Now, let $t^ \ast$ be so small that  $C(\sqrt{t^\ast}+t^\ast)\le \frac 12 $. Then, we get
$$
\EE \sup_{0\le s\le t^ \ast}  \lk[ \CH  (u(s))-\CH  (u(0))\rk]\le 2 C,
$$
and therefore
$$
\EE \sup_{0\le s\le t^ \ast}   \CH  (u(s))\le \EE \CH  (u(0))+ 2C.
$$
Noting that  $t^\ast$ is  only depending on $C_0(\nu)$, $C_1(\nu)$, $C_2(\nu)$,  and $\alpha$, one may iterate
the previous step on $[t^\ast, 2t^\ast]$, etc, and show that \eqref{energybound} holds.


%

%
\medskip

Next, we want to investigate the entity $\int_{\RR^d } x^2 |u(t,x)|^2 \, dx$.
First, we will  prove  the inequality
$$
\EE \int_{\RR^d } x^2 |u(t,x)|^2 \, dx \le C(T)\left[1+  \EE(\CH(u_0)) + \EE  \int_{\RR^d } x^2 |u(0,x)|^2 \, dx \right] ,
$$
where the constant $C>0$ depends only on $C_0(\nu)$, $C_1(\nu)$, $C_2(\nu)$, and $\alpha$.
Secondly, we will give an estimate of $\EE\sup_{0\le s\le T} \int_{\RR^d } x^2 |u(s,x)|^2 \, dx$.
In particular, we have by the It\^o formula
\DEQS
\lqq{ d \int_{\RR^d} |x|^ 2 |u (t,x)|^2 \, dx
}
&&
\\  &=&
 \,2\langle x^2 u(t),du^ c(t)\rangle
 \\&&+\int_Z \lk[x^2|u(t^-)-iu(t^-)G(z)|^2-x^2|u(t^-)|^2\rk] \, \tilde \eta(dz, dt)
 \\&&+\int_Z \lk[x^2|u(t^-)-iu(t^-)G(z)|^2-x^2|u(t^-)|^2- 2\la x^ 2 u(t^ -) ,u(t^ -) (-iG(z))\ra\rk] \,  \nu(dz) \,dt
\\
&=& 2\langle x^2 u, -i\Delta u+i\lambda |u|^{\alpha-1}u\rangle\, dt -2\int_Z \langle x^2 u, i u \CCC(z)\,\rangle\nu(dz)\, dt \\
&&  + \int_Z \lk(\int_{\RR^d} x^2|u(t^-,x)|^2 \lk[ |1-ig(z(x))|^2-1\rk] \,dx\rk) \,\tilde\eta(dz,dt)
 \\&&+\int_Z \int_{\RR^ d } x^2\Big\{ |u(t^-,x)[1-ig(z(x))]|^2-|u(t^-,x)|^2
 \\
 &&{}- 2\Re\lk[  x^2u(t^ -,x) \overline{ u(t^ -,x) (-ig(z(x)))} \rk] \Big\}\, \nu(dz) \,dt
.
\EEQS
Integrating by parts,
\DEQS
\lqq{-\langle x^2u, i\Delta u\rangle =\langle \nabla (x^2 u), i\nabla u\rangle}&&
\\
&=&2 \langle xu, i\nabla u\rangle =-2 \Im \int_{\RR^d} \bar u(x) x \cdot \nabla u(x) \,dx.
\EEQS
Next,
\DEQS
 \langle x^2 u, i u \CCC(z)\,\rangle
 &=&
 \int_{\RR^ d } \Re\lk( x^2 u(t,x)\overline{u(t,x) i \hh (z(x)))}\rk)\,dx\,
\\
 &=&-
 \int_{\RR^ d } |x|^ 2 |u(t,x)|^ 2 \Im\lk(   \hh (z(x)))\rk)\,dx.
\EEQS
Since, for any complex valued functions $a$ and $b$ we have $|a| ^2 -|b| ^2 =\Re\la a-b,a+b\ra $, we have
\DEQS
\lqq{\int_{\RR^ d } x^2\lk[|u(t^-,x)[1-ig(z(x))]|^2-|u(t^-,x)|^2- 2 \Re \lk(u(t^-,x)\overline{u(t^-,x)(-ig(z(x)))} \rk) \rk]  \, dx
}&&
\\
&=&\int_{\RR^ d } x^2\lk[ \Re\lk(u(t^-,x)( -ig(z(x))\overline {u(t^-,x)(2-ig(z(x)))}\rk) - 2 |u(t^-,x)|^ 2 \Re \lk(\overline{-ig(z(x))}\rk)  \rk]  \, dx
\\
&=&  \int_{\RR^ d } x^2|u(t^ -,x)|^ 2 \lk(2\Im \lk(g(z(x))\rk)  - 2  \Re \lk(i\overline{g(z(x))}\rk) + |g(z(x))|^ 2 \rk)\, \, dx
\\
&=&  \int_{\RR^ d } x^2|u(t^ -,x)|^ 2  |g(z(x))|^ 2  \, dx.
\EEQS
Collecting altogether, taking expectation we get
\DEQS
\EE  \int_{\RR^ d} x^ 2 |u(s,x)|^ 2 \, dx
&\le & - 4\Im \int_0^ t  \EE\int_{\RR^d } \bar u(s,x) x\cdot \nabla u(s,x) \, dx\, ds
\\
&&{} +  2\int_0^ t  \int_Z \EE \int_{\RR^ d } |x|^ 2  |u(s,x)|^ 2 \Im\lk(   \hh (z(x)))\rk)\,dx\, \nu(dz)\, ds
\\
&&{} +  \int_0^ t  \int_Z \EE \int_{\RR^ d } |x|^ 2  |u(s,x)|^ 2   |g(z(x))|^ 2 \,dx\, \nu(dz)\, ds
.
\EEQS
Taking into account Hypothesis \ref{hypogeneral}-(ii)-(b), we get by the Young inequality 
\DEQS
\lqq{ \EE \int_{\RR^ d} x^ 2 |u(s,x)|^ 2 \, dx  }
\\
&\le&C(\nu) \, \EE\int_0^ t  \int_{\RR^d } x^ 2  |{u(s,x)}|^ 2 \, dx \, ds+ 2 \EE \int_0^ t \int_{\RR^d } |\nabla  u(s,x)|^ 2  \, dx\, ds.
\EEQS
Since the second term is bounded by $\EE \CH(u(t))$, the assertion follows by the Grownwall inequality.
Next, observe that
\DEQS
\lqq{\int_{\RR^ d } x^2\lk[|u(t^-,x)[1-iG(z)]|^2-|u(t^-,x)|^2\rk]  \, dx
}&&
\\
&=&  \int_{\RR^ d } x^2|u(t^ -,x)|^ 2 \lk(2\Im (g(z(x)) + |g(z(x))|^ 2 \rk)\, \, dx .
\EEQS
Hence, we get in addition, by Burkholder inequality
\DEQS
\lqq{\EE \sup_{0\le s\le t}  \int_{\RR^ d} x^ 2 |u(s,x)|^ 2 \, dx  }
&&
\\
&\le & - 4\Im \int_0^ t  \int_{\RR^d }\EE \bar u(s,x) x\cdot \nabla u(s,x) \, dx\, ds
\\
&&{} +  \int_0^ t \int_Z \int_{\RR^ d } |x|^ 2 \EE |u(s,x)|^ 2 \Im\lk(   \hh (z(x)))\rk)\,dx\,\nu(dz)\, ds
\\
&&{}+\EE \sup_{0\le s\le t} \int_0^ t \int_Z  \int_{\RR^ d }x^2\lk[|u(t^-,x)[1-G(z)]|^2-|u(t^-,x)|^2\rk] \, dx \,\eta(dz,ds)
\\
&\le & -  4\Im \int_0^ t  \int_{\RR^d }\EE \bar u(s,x) x\cdot \nabla u(s,x) \, dx\, ds
\\
&&{} +  \int_0^ t \int_Z \int_{\RR^ d } |x|^ 2 \EE |u(s,x)|^ 2 \Im\lk(   \hh (z(x)))\rk)\,dx\,\nu(dz)\, ds
\\
&&{}+\int_0^ t \int_Z  \int_{\RR^ d } x^2 \EE |u(s,x)|^ 2  \lk(2\Im (g(z(x)) + |g(z(x))|^ 2 \rk)\, \, dx\,\nu(dz)\, ds.
\EEQS
By similar arguments as before, we can show that
\DEQS
\EE \sup_{0\le s\le t}  \int_{\RR^ d} x^ 2 |u(s,x)|^ 2 \, dx &\le & C(T)\lk( \EE \int_{\RR^ d} x^ 2 |u(0,x)|^ 2 \, dx +1+  \EE |\nabla u(t)|^2\, dx \rk) .
\EEQS

\medskip

\end{proof}

\section{Existence  of the solution with infinite \levy measure - Proof of Theorem \ref{main}}\label{reg}
\label{main_proof}

The proof is done in several steps. In the first step we  construct a solution by cutting of the small jumps.
In this way, we get a sequence of  solutions, denoted in the following by $\{u_m:m\in\NN\}$.
Next, in the second step, we give uniform bounds on the mass $\CE(u_m)$, the Hamiltonian $\CH(u_m)$
and the virial $|xu_m|_{L^2}^2$.
Thanks to these uniform bounds we are able to prove in the third step tightness of the laws of
$\{u_m:m\in\NN\}$ in $\DD(0,T;H^\gamma_2(\RR^d ))$ for any $\gamma<1$.
Now, the existence of a converging  subsequence follows.
In order, to get again stochastic processes   we apply  the Skorohod embedding Theorem.
 This gives us a probability space with
a family of processes converging in the almost sure sense.
Now in the last step we can show by an application of the dominated convergence Theorem that
this limit is indeed a solution to \eqref{itoeqn}.

\medskip

\paragraph{\bf Step I}{In the first step we will construct an approximating sequence.}
Let $\{ \ep_m:m\in\NN\}$ be a sequence such that $\ep_m>0$ and $\ep_m \downarrow 0$.
Let  $\nu_m$ be the \levy measure defined by $\nu_m(U):= \nu( U \setminus B_Z(\ep_m))$, $U\in\CB(Z)$.
Let $\eta$ be a time homogenous Poisson random measure over a filtered probability  space $\mathfrak{A}$ and let $\eta_m$ be the time homogenous random measure
given by
\DEQSZ\label{defprmapp}
\eta_m(U\times I) := \eta( U\setminus B_Z(\ep_m)\times I),\quad U\in\CB(Z),\,I\in \CB([0,T]).
\EEQSZ
Let us observe that $\eta_m$ has intensity measure $\nu_m$.
We denote by $u_m$ the solution to
\DEQSZ
\label{oben_m} 
u_m(t) &= &\CT(t) u_0+ i\lambda\int_0^ t \CT(t-s)
 \lk|u_m(s)\rk|^{\alpha-1} \,
 u_m(s)\, ds
\\\nonumber
&&{}-i \int_0^t\int_Z  \CT(t-s) u_m(s)\, G(z) \, \tilde \eta_m(dz,ds)
\\ \nonumber
&&{}-i\int_0^ t \int_{Z} \CT(t-s)u_m(s) \, \CCC(z)\, \nu_m(dz)\, ds.
\EEQSZ
It follows from  Lemma  \ref{local_ex} that for any $m$ there exists a unique 
solution with $u_m$  to Equation \eqref{oben_m} belonging $\PP$-a.s.\ to $\DD(0,T;H^1(\RR^d ) )$.

\medskip

\paragraph{\bf Step II}
We will prove the following Claim.
\begin{claim}\label{first_uniform}
\begin{itemize}

  \item For any $T>0$ there exists a constant $C=C(T,C_0(\nu),C_3(\nu),C_g,C_h)>0$ such that
  $$\EE \sup_{0\le t\le T} |u_m(t)|_{L ^ 2}= C\,\lk( |u_0|_{L ^ 2}+1\rk) ,\quad m\in\NN,\, \PP-\mbox{a.s.}
$$
  \item For any $T>0$  there exists a $C=C(T,C_0(\nu),C_1(\nu),C_2(\nu),C_g,C_h)>0$ such that
  $$
\EE\sup_{0\le s\le T}  \CH(u_m(s))  =C\,\lk( \EE \CH(u(0))+1\rk) ,\quad m\in\NN.
$$
  \item For any $T>0$ there exists a $C=C(T,C_0(\nu),C_1(\nu),C_2(\nu),C_g,C_h)>0$ such that
  $$
\EE\sup_{0\le s\le T} \int_{\RR^ d} |x|^ 2 |u(s,x)|^ 2 \, dx   = C\,\lk(  \EE\int_{\RR^ d} |x|^ 2 |u(0,x)|^ 2 \, dx +\CH(u(0))+1\rk),\quad m\in\NN.
$$

\end{itemize}
\end{claim}
\begin{proof}
In fact the proof of Proposition \ref{first_uniform} is just an application of Proposition \ref{local_ex} and taking into account, that
$$
C_j(\nu_m) \le C_j(\nu),\quad m\in\NN \mbox{ and } j=0,1,2\, \mbox{and}\, 3.
$$
\end{proof}

\paragraph{\bf Step III} In this Step we show the following Claim.
\begin{claim}\label{cseq}
\begin{enumerate}
  For any $\gamma<1$ the laws of the set $\{ u_m:m\in\NN\}$ are  tight in $\DD(0,T;H^ \gamma_2(\RR^d))$
\end{enumerate}
\end{claim}
\begin{proof}
In order to show the assertion, we will apply Corollary \ref{comp-2}.
We will first  prove the compact containment condition, in particular, the condition (a) in Corollary \ref{comp-2}.
Let $B_0$ be defined by
$$
\lk\{v\in H ^1_2(\RR^d): \int_{\RR ^d} |x| ^2 |v(x)| ^2 \, dx<\infty\rk\}
$$
equipped with norm
$$ |v|_{B_0}:=|v|_{H ^1_2} + \lk( \int_{\RR ^d} |x| ^2 |v(x)| ^2 \, dx\rk) ^\frac 12
.
$$
Then,  $B_0\hookrightarrow L^ 2 (\RR^d)$ compactly.
Let us denote the complex interpolation space $B_1=(B_0,H_2^1(\RR^d))_\delta$.
Then, the embedding of $B_0$ into $H^1_2(\RR^d)$ is bounded, and the embedding of $B_0$ into $L^2(\RR^d)$ is compact.
By Theorem 3.8.1 \cite[p.\ 56]{bergh} it follows that for any $\delta<1$, $B_0\hookrightarrow H^\delta _2(\RR ^d)$ compactly.
To show condition (a) of Corollary \ref{comp-2}, i.e.\ the compact containment condition,
one can use Claim \ref{first_uniform}. To be more precise, it follows from item one, two and three of Claim \ref{first_uniform} that there exists a constant $C>0$ such that
$$
\EE |u_m(t)|_{B_0}\le C,\quad m\in\NN.
$$
Condition (a) follows by the Chebyschev inequality.

 It remains to show, that the family
$\{u_m:m\in\NN\}$ satisfies the second condition of Corollary \ref{comp-2}, i.e.\ (b).
First, observe that we have
\DEQS \lqq{
u_m(t+h)-u_m(t)= \lk[ \CT(h)-I\rk] \CT(t) u_0 +i \lambda  \int_t^ {t+h} \CT( t+h-s)  F(u_m(s)) \, ds } &&
 \\&&+ i \lambda
 \lk[ \CT(h)-I\rk] \int_0^ {t} \CT( t-s)F(u_m(s))  \, ds
 \\
&& {} -i \int_t^ {t+h}\int_Z  \CT(t+h -s) u_m(s)\, G(z) \tilde \eta_m(dz,ds)
\\
&&{} -i \lk[ \CT(h)-I\rk] \int_0^ {t}\int_Z  \CT(t -s) u_m(s)\, G(z) \tilde \eta_m(dz,ds)
 \\
&& {}-i \int_t^ {t+h}\int_Z  \CT(t+h -s) u_m(s)\, \CCC(z)  \nu_m(dz)\,ds
 \\
&& {}- i \lk[ \CT(h)-I\rk]  \int_0^ {t}\int_Z  \CT(t -s) u_m(s)\, \CCC(z)  \nu_m(dz)\,ds
\EEQS
\DEQS
&=&  \underbrace{\lk[ \CT(h)-I\rk] u_m(t)}_{=:I_0(t,h)}   +  \underbrace{ i \lambda \int_t^ {t+h} \CT( t+h-s)  F(u_m(s)) \, ds}_{=:I_1(t,h)}
\\
&&\quad {}- \underbrace{ i \int_t^ {t+h}\int_Z  \CT(t+h -s) u_m(s)\, G( z) \tilde \eta_m(dz,ds)}_{=:I_2(t,h)}
\\
&& \quad {} - \underbrace{ i \int_t^ {t+h}\int_Z  \CT(t+h -s) u_m(s)\, \CCC(z)  \nu_m(dz)\,ds}_ {=:I_3(t,h)}
.
\EEQS
Secondly, note that
\DEQS
\lqq{\EE\sup_{0\le h\le \delta} | u_m(t+h)-u_m(t)|^r_{H^ \gamma_2}}
&&\\
&\le &\EE\sup_{0\le h\le \delta} | u_m(t+h)-u_m(t)|_{L^2}^ {r(1-\gamma)} | u_m(t+h)-u_m(t)|_{H^1_2}^ {r\gamma}
\\
&\le&\EE\sup_{0\le h\le \delta} | u_m(t+h)-u_m(t)|_{L^2}^  {r(1-\gamma)} \lk(| u_m(t+h)|_{H^1_2}+|u_m(t)|_{H^1_2}\rk)^{r\gamma}
\\
&\le &\EE\sup_{0\le h\le \delta} | u_m(t+h)-u_m(t)|_{L^2} ^ {r(1-\gamma)} \EE\sup_{0\le h\le T} \lk(| u_m(t+h)|_{H^1_2}+|u_m(t)|_{H^1_2}\rk)^ {r\gamma}
\\
&\le &\lk( \EE\sup_{0\le h\le \delta} | u_m(t+h)-u_m(t)|^ r_{L^2}\rk)  ^{1-\gamma} \lk( \EE\sup_{0\le h\le T} \lk(| u_m(t+h)|_{H^1_2}+|u_m(t)|_{H^1_2}\rk)^ r \rk) ^ {\gamma}.
\EEQS

In order to estimate $I_0(t,h)$  we know for any $s\in\RR$ $\lk| \lk[ \CT(h)-\mathcal{I}\rk] u_m(t) \rk|_{H^s_ 2} \le h \lk| u_m(t)\rk|_{H^{s+2} _2}$.
Interpolation gives therefore
$$
\lk| \lk[ \CT(h)-\mathcal{I}\rk] u_m(t)\rk|_{L^ 2} \le
\sqrt{2 h}  \lk| u_m(t)\rk|_{H^ 1 _2}.
$$
Since by Claim \ref{first_uniform}, $\EE\,\CH(u_m(t))$ is uniformly bounded in $m$, there exists a constant $C=C(T)>0$ such that
$$
\EE \sup_{0\le h\le \delta}\lk| \lk[ \CT(h)-\mathcal{I}\rk] u_m\rk|^ 2_{L _2} \le C
\sqrt{ \delta }  .
$$

Let $p=\alpha+1$, $p'=(\alpha+1)/\alpha$, $r=4(1+\alpha)/(\alpha-1)d$ and $r'=4(1+\alpha)/(4(1+\alpha)+d(1-\alpha))$.
Applying the Strichartz estimate, we get for  $I_1(t,h)$
\DEQS
\sup_{0\le h\le \delta} \lk| I_1(t,h) \rk|_{L^ 2} 
&\le & \lambda \sup_{0\le h\le \delta} \lk| \int_t^ {t+h} \CT(t+h-r)  F(u_m (r))\, dr  \rk|_{L^2} 
\\
&\le &C\, \lk(\int_t^ {t+\delta} \lk|  F(u_m (s))\rk| ^{r'}_{L^ {p'} }ds\rk)^{1/r'}
\le C  \lk(\int_t^ {t+\delta} \lk|  u_m (s)\rk| ^{r'\alpha }_{L^ {\alpha+1} } ds\rk)^{1/r'}.
\EEQS
Sobolev embedding gives for $\alpha< (d+2)/(d-2)$
\DEQS
\sup_{0\le h\le \delta} \lk| I_1(t,h) \rk|_{L^ 2} &\le & C \lk( \int_t^ {t+\delta} \lk|  u_m (s)\rk| ^{r'\alpha }_{H^ 1_2} ds\rk)^{1/r'}.
\EEQS
Now, taking  the expectation, the H\"older inequality gives
\DEQS
\EE \sup_{0\le h\le \delta} \lk| I_1(t,h) \rk|_{L^ 2}^ {2\over \alpha} 
&\le & C\, \delta^{2/ r'\alpha} \, \EE \sup_{t\le h\le t+\delta} \lk|  u_m (h)\rk|_{H^ 1 _2 } ^{2}
.
\EEQS
By Hypothesis \eqref{hypogeneral}-(i)-(a), it follows   for $I_2(h,t)$
\DEQS
\EE \sup_{0\le h\le \delta} \lk| I_2(t,h) \rk|_{L^2}^2 &\le & \EE \sup_{0\le h\le \delta} \lk| \int_t^ {t+h}\int_Z \CT(t+h-s)  u_m (s)G(z)\, \tilde \eta(dz, ds)  \rk|_{L^2}
\\
&\le& \EE  \int_t^ {t+\delta}\int_Z \lk|\CT(t+h-s)  u_m (s) G(z)\,\rk|^ 2 _{L^2}\nu(dz)\, ds
\\
&\le & C\, \delta \sup_{t\le s\le t+\delta}  \EE\lk|  u_m (s)\rk|^2 _{L^ 2 }
.
\EEQS
Again, by Hypothesis \eqref{hypogeneral}-(i)-(a) and (iii), we get by Theorem \ref{RS2} for  $I_3(h,t)$
\DEQS
\EE \sup_{0\le h\le \delta} \lk| I_3(t,h) \rk|_{L^2} &\le & \EE \sup_{0\le h\le \delta} \lk| \int_t^ {t+h}\int_Z \CT(t+h-s)  u_m (s)H(z) \,  \nu(dz)  ds  \rk|_{L^2}
\\
&\le& \EE  \int_t^ {t+\delta}\int_Z \lk|\CT(t+h-s)  u_m (s)H(z) \,\rk|_{L^2}\nu(dz)\, ds
\\
&\le & C\, \delta \sup_{t\le s\le t+\delta}  \EE\lk|  u_m (s)\rk|^2 _{L^2}
.
\EEQS
Collecting altogether we arrive at
\DEQS
\lqq{ \EE \lk( \sup_{t\le s\le h} \lk| u_m(t+s)-u_m(t)\rk|_{L^ 2 } ^ {2\over \alpha}\rk)   } &&
\\ &\le &  \EE\lk( \sup_{0\le s\le h}  \lk|\lk[ \CT(s)-I\rk] u_m(t)\rk|_{L^ 2 }\rk) ^ {2\over \alpha } + \EE \lk( \lk| I_1(t,h) \rk|_{L^ 2 }\rk) ^ {2\over \alpha }
+\EE \lk( \lk| I_2(t,h) \rk|_{L^ 2 }\rk) ^ {2\over \alpha}
\\
&\le & h ^\frac 1{\alpha }  \lk(1+ \EE \lk|  u_m(t)\rk|_{H^1_2}^{2}\rk),
\EEQS
since $r'\le 2$.
Applying Corollary \ref{comp-2}, we know the family of laws of $\{ u_m:m\in\NN\}$  is tight in $\DD(0,T;H^{\gamma}_2(\RR^d))$. 

\medskip

\paragraph{\bf Step IV:}
First, note, since $ \nu_m\to \nu$ on $Z\setminus \{0\}$, hence $\eta_m\to \eta$ in $M(Z\setminus\{0\}\times [0,T])$, so the family $\{ \eta_m:m\in\NN\}$ is tight in $M(Z\setminus\{0\}\times [0,T])$.
Now, it follows from Step III, i.e.\ the fact that the sequence $\{u_m:m\in\NN\}$ is tight in $\DD(0,T;H^ \gamma_2(\RR^d))$,
that there exists a pair $(u^\ast,\eta^\ast)$ of $\DD(0,T;H_2^ \gamma(\RR^d )) \times M_I(Z\setminus\{0\}\times [0,T])$--valued random variables over $\mathfrak{A}$
and a subsequence  $\{ m_k:k\in\NN\}$ 
such that $\{(u_{m_k},\eta_{m_k}):k\in\NN\}$ converges to $(u^\ast,\eta^\ast)$ weakly in $M_I(Z\setminus\{0\}\times [0,T])\times\DD(0,T;H_2^ \gamma(\RR^d ))$.
In fact, by the construction of $\eta_m$ we have even $\eta^\ast=\eta$.
For simplicity, we denote the subsequence $\{(u_{m_k},\eta_{m_k}):k\in\NN\}$ again by $\{(u_{m},\eta_{m}):m\in\NN\}$.
By  the  modified version of the Skorohod embedding Theorem, see Theorem D.1 \cite{reactdiff},
there exists a probability space $(\bar \Omega,\bar \CF,\bar \PP)$
and $\DD (0,T;H^\gamma_2(\RR^d ))\times M_{I}(Z\setminus\{0\}\times[0,T])$-valued random variables $(\bar u_1,\bar \eta_ 1)$, $(\bar u_2,\bar \eta_2)$, $\ldots $, having the same law as
the random variables $( u_1, \eta_ 1)$, $( u_2,\eta_2)$, $\ldots $, and a $ \DD(0,T;H^\gamma _2(\RR^d ))\times M_I(Z\times \RR_+)$-valued random variable
$(\bar u^\ast,\bar \eta^\ast)$ on $(\bar \Omega,\bar \CF,\bar \PP)$ with $\CL((\bar u^\ast, \bar \eta^\ast))=\mu^\ast$
such that $\PP$ a.s.
\DEQSZ\label{limita-s}
\lqq{ 
\lk( \bar u _m,\bar \eta_ m \rk) \longrightarrow  \lk(\bar  u ^\ast,\bar \eta ^\ast\rk) } &&
\\&&\nonumber
 \mbox{ in }   \DD(0,T;{H^{\gamma} _2(\RR^d) }) \times M_I(Z\setminus\{0\}\times \RR_+). 
\EEQSZ

Before continuing,  we will introduce the following notation.
For a random measure $\mu$ on $Z\setminus\{0\}\times [0,T]$ and for any $U\in {\mathcal{B}(Z\setminus\{0\})} $ 
let us define  an  $\bar{\mathbb{N}}$-valued process
$(N_\mu(t,U))_{t\ge 0}$ by
$N_\mu(t,U):= \mu(U \times (0,t]), \;\; t\ge 0.
$
In addition, we denote by  $(N_\mu(t))_{t\ge 0}$ the measure
valued process defined by $N_\mu(t) =\{ {\mathcal{S}} \ni U \mapsto
N_\mu(t,U)\in\bar{\mathbb{N}}\}$, $t\in[0,T]$.

Now, let
$\bar {\mathbb{F}}=(\bar {\mathcal{F}}_t)_{t\ge 0}$
be the
filtration defined for any $t\in[0,T]$ by
\begin{equation} \label{eq:filtrat}
 \bar {\mathcal{F}}_t= \sigma( \{ (N_{\bar\eta_m}1_{[0,s]} ,\bar u_m1_{[0,s]}), m\in{\mathbb{N}}\},\, (N_{\bar \eta^ \ast}1_{[0,s]},\bar u^\ast1_{[0,s]}); 0\le s\le t).
\end{equation}
Since $\sigma(N_{\bar \eta_m}1_{[0,s]})\subset \sigma(N_{\bar \eta_{m+1}}1_{[0,s]})$,  it is easy to show that the filtration obtained by
deleting the family  $\{\bar \eta_m:m\in\NN\}$ in
\eqref{eq:filtrat} is the equal to $\bar {\mathbb{F}}$.

\begin{claim}\label{ersterclaim}
The following holds
\begin{enumerate}
\item for every $m\in\NN$, $\bar \eta_m$ is a time homogeneous Poisson random measure on $\CB(Z)\times \CB(\RR_+)$
over $(\bar \Omega,\Bar \CF,\bar \BF ,\bar \PP)$ with intensity measure $\nu_m$;
\item $\bar \eta^ \ast$ is a time homogeneous Poisson random measure on $\CB(Z)\times \CB([0,T])$
over $(\bar \Omega,\Bar \CF,\bar \BF,\bar \PP)$ with intensity measure $\nu$;
\end{enumerate}
\end{claim}
Before starting with the actual proof,
 we  cite the following Lemma.
\begin{lemma}\label{eqdefprm} A measurable mapping $\eta:\Omega\to M_{I}(Z\times R^+_0) $ is a time homogeneous Poisson random measure with intensity  $\nu$ iff
\begin{letters}\label{prmpoints}
\item for any $\AAA \in\CZ$ with $\nu(U)<\infty$, the random variable $N_\eta (t,\AAA ) $ is
Poisson distributed with parameter $t\,\nu(\AAA )$;
\item for any
disjoint sets $\AAA _1,\AAA _2,\ldots,\AAA _n\in\CZ$, and any $t\in[0,T]$ the random
variables $N_\eta (t,{\AAA _1}) $, $N_\eta (t,{\AAA _2}) $, \ldots, $N_\eta (t,{\AAA _n}) $  are  mutually independent;
\item the $M_{\bar {\Bbb N}}(Z)$-valued process  $(N_\eta (t,\cdot))_{t\ge 0}$ is adapted to $\BF$; 
%
\item for any
$t\in[0,T]$, $\AAA \in\CZ$, $\nu(U)<\infty$, and any $r,s\ge t$, the random variables
$N_\eta (r,\AAA )-N_\eta (s,\AAA )$ are independent of $\CF_t$.
\end{letters}
\end{lemma}

\begin{proof}[Proof of Claim \ref{ersterclaim}]
We have to show that for arbitrary $m\in\NN$, $\bar \eta_m$ satisfies item (a), (b), (c) and (d).
In order to show (a), let $\AAA\in \mathcal{B}(Z)$. Then $N_{\eta_m}(\AAA,t)$ is Poisson distributed with parameter $t\nu_m(\AAA)$. Since $\Law(\eta_m)=\Law(\bar \eta_m)$, it follows (a).
In order to show (b),  let $U_1,\ldots, U_k\in{\mathcal{Z}}$, be $k$ disjoint sets and $t\ge 0$. Since $\eta_m$ is a time homogeneous Poisson
 random measure, we have for all $k\ge 1$ and $\theta_j\in\RR$, $j=1,\ldots k$
 \begin{eqnarray}\label{KIK-FUNC} {\mathbb{E}}
 e ^{i\left(\sum_{l=1}^k \theta_l  N_{\eta_m} (t,U_l)  \right)} =
\prod_{l=1}^k  {\mathbb{E}} e ^{i\;\theta_l  N_{ \eta_m} (t,U_l) }.
 \end{eqnarray}
  Since $\bar\eta_m$ and $\eta_m$ have the same laws, the random variables $N_{\bar  \eta_m} (t,U)$ and $N_{ \eta_m} (t,U)$
 have the same characteristic functions for any $ U\in \mathcal{Z}$ and $t\ge0$. Therefore, it follows from \eqref{KIK-FUNC} that
 \begin{eqnarray*}
 {\bar{ \mathbb{E}}}
  e ^{i\left(\sum_{l=1}^k \theta_l  N_{\bar \eta_m} (t,U_l)   \right)} =
\prod_{l=1}^k    \bar{\mathbb{E}} e ^{i\;\theta_l  N_{ \bar \eta_m} (t,U_l) }
,
  \end{eqnarray*}
which proves (b).
Next, we have to show that $\bar \eta_m$ satisfies (c)
  with the filtration defined in \eqref{eq:filtrat}.
  For this purpose let us fix $l\in \mathbb{N}$, $t_0\in [0,T]$ and $r\ge s\ge t_0$.
  It follows from the definition of $\bar{{\mathbb{F}}}$ that $N_{\bar \eta_m}$ is $\bar{{\mathbb{F}}}$-adapted.
  It remains to prove (d). In particular, it remains to prove that the random variable
  $\bar X= N_{\bar \eta_m}(r)-N_{\bar \eta_m}(s)$ is independent of
  $\bar{\mathcal{F}}_{t_0}$.
  By Lemma \ref{defprmapp} the random variable $ X= N_{ \eta_m}(r)-N_{ \eta_m}(s)$ is independent of
  $N_{\eta}(t_0)$. Since for any $k\ge 1$, the $\sigma$--algebra generated by $\eta$ until time $t_0$ is finer than the $\sigma$--algebra generated by $\eta_{k}$,
  we know $X$ is independent from  $N_{\eta_k}(t_0)$ for all $k\in\NN$. In particular, for all
  $f,g:M_I(Z\times [0,T])\to \RR$
  we have for all $k\ge 1$
  $$
  \EE f(X)g( N_{\eta_k}(t_0))=\EE f(X)\, \EE g(N_{\eta_k}(t_0)).
  $$
Since for all $k\ge 1$, $\bar \eta_ k$ have the same law as $\eta_k$, and $\bar \eta_ m$ have the same law as $\eta_m$, therefore $X$ has the same law as $\bar X$.
It follows that
  $$
\bar   \EE f(\bar X)g( N_{\bar \eta_k}(t_0))=\EE f(\bar X)\, \EE g(N_{\bar \eta_k}(t_0)).
  $$
Hence, $\bar X$ is independent of the filtration $\sigma\lk (  N_{\bar \eta_k}(t),\, t\le t_0 \rk)$.

Next, we need to show that for any $k\ge 1$,  $\bar X$ is independent of $\bar u_k(t)$ for any $t\le t_0$.
In what follows we also fix $t \in [0,t_0]$. Since $\mathcal{L}(\bar{u}_l, \bar{\eta}_l)=\mathcal{L}(u_l, \eta_l)$, it follows that
  \begin{equation}\label{EQ-LAW}
\mathcal{L}(\bar{u}_l|_{[0,t]}, \bar{X}_l)=\mathcal{L}(u_l|_{[0,t]}, X_l),
  \end{equation}
 where $X=N_{\eta_m}(r)-N_{\eta_m}(s)$.
Now, we have to show that $\bar X$ is independent to the $\sigma(\{ \bar u_k|_{[0,t]} : t\le t_0, k\in\NN\})$.
Recall that  $u_k$ is the unique solution to the linear stochastic evolution equation \eqref{eq_1}, hence it is adapted to
the $\sigma$-algebra generated by
$\eta_k$. Consequently, $u_k|_{[0,t]}$ is independent of $X$ and we infer from this last remark and the equality of the laws  that $\bar{u}_k|_{[0,t]}$ is independent of
$\bar{X}$ for all $t\le t_0$ and $k\in\NN$.

It remains to prove that $\bar{X}$ is independent of $\sigma(\{ N_{\bar \eta^\ast}(t),t\le t_0\} )$ and $u^ \ast|_{[0,t]}$, but, this is the object
of  the next Lemma.
\begin{lemma}\label{independenc}
Let $Y$ be a Banach space, $z$ and $y_\ast$ be two $Y$-valued random
variables over $(\Omega,{\mathcal{F}},\mathbb{P})$. Let $\{y_n:n\in{\mathbb{N}}\}$  be a
family of $Y$-valued random variables over a probability space
$(\Omega,{\mathcal{F}},\mathbb{P})$ such that $y_n\to y_\ast$ weakly, i.e.\ for all
$\phi\in Y ^ \ast$, $ \mathbb{E} e ^ {i\la \phi,y_n\ra }\to \mathbb{E}   e ^ {i\la
\phi,y\ra }$.
If for all $n\ge 1$ the two random variables $y_n$ and $z$ are
independent, then $z$ is also independent of $y_\ast$.
\end{lemma}
\begin{proof}[Proof of Lemma \ref{independenc}]
The random variables $y_\ast$ and $z$ are independent iff \begin{eqnarray*}
{\mathbb{E}} e ^{i(\la\theta_1, z\ra + \la\theta_2, y_\ast\ra )} = {\mathbb{E}} e ^{i\la\theta_1,
z\ra}\,{\mathbb{E}} e^{ i\la\theta_2, y_\ast\ra }, \quad \theta_1,\theta_2\in Y  ^
\ast. \end{eqnarray*}
The weak convergence and the independence of $z$ and $y_n$ for all $n\in{\mathbb{N}}$ justify the following chain of equalities.
\begin{eqnarray*} {\mathbb{E}} e ^{i(\la\theta_1, z\ra + \la\theta_2, y_\ast\ra )} =
\lim_{n\to\infty}{\mathbb{E}} e ^{i(\la\theta_1, z\ra + \la\theta_2, y_n\ra )} =
\lim_{n\to\infty}{\mathbb{E}} e ^{i\la\theta_1, z\ra}\,{\mathbb{E}} e^{ i\la\theta_2, y_n\ra}={\mathbb{E}} e
^{i\la\theta_1, z\ra}\,{\mathbb{E}} e^{ i\la\theta_2, y_\ast\ra}. \end{eqnarray*}
\end{proof}
Fix $t\le s$. Since $\bar u_k|_{[0,t]}$ is independent
from $ \bar{X}$, Lemma \ref{independenc} implies that $u^\ast|_{[0,t]}$ is independent from $\bar{X}$.
Similarly, $ \bar{X}$ is independent from $N_{\eta^ \ast}(t)$ for all $t\le t_0$.

Finally we have to show Claim \ref{ersterclaim}-(b). In particular, we have to show that
$\eta_\ast\in{\mathcal{M}}_I(S\times {\mathbb{R}} ^ +)$ is a time homogeneous Poisson random measure
with intensity $\nu$. Observe first that  $Z\setminus B(\ep_m) \uparrow Z\setminus\{0\}$ as $m\to\infty$. By Theorem 4.3.4 \cite{applebaum} one knows  $\eta_m\to \eta$ weakly. Using the fact that $\eta$ is a Poisson random measure with intensity measure $\nu$ and $\Law(\bar \eta_m)=\Law(\eta_m)$ and $\Law(\eta^ \ast)=\Law(\eta)$, the assertion follows by Lemma \ref{independenc}.
\end{proof}

\begin{claim}\label{zweiter claim}
 The following  holds
 \begin{enumerate}
  \item  for all $m\in\NN$, $\bar u_m$ is a  $\bar\BF$-progressively measurable process; 
\item
the process $ \bar u^\ast$ is a
$\bar\BF $
-progressively measurable process.
\end{enumerate}

\end{claim}

\begin{proof}
As we noted earlier, one can argue as in \cite[Proposition B.5]{uniquness2} and prove that
the random variables $\bar{u}_m, \bar u^\ast: \bar{\Omega} \to
 \DD(0,T;H^\gamma _2(\RR^d ))$ induce two $H^ \gamma_2(\RR^d)$-valued stochastic processes still
denoted with the same symbols. Here, we have to show that for each
$m\in{\mathbb{N}}$, $\bar u_m$ and
$\bar u^\ast$ are
$\bar {\mathbb{F}}$-progressively measurable. By definition of $\bar {\mathbb{F}}$,
for fixed $m\in{\mathbb{N}}$ the process $\bar u_m$ is adapted to $\bar
{\mathbb{F}}$ by the definition of $\bar {\mathbb{F}} $.
By Step (III) the laws of the processes  $\{\bar u_m:m\in\NN\}$ are tight in $\DD(0,T;H^\gamma _2(\RR^d ))$.
Thus, for any $\ep>0$ there exists a compact set $K_\ep\subset H^\gamma _2(\RR^d )$
such that $\PP\lk( \bar u_m\not\in K_\ep\rk)<\ep$. However, by Proposition B.5
the Haar projections $\pi_n$ converge uniformly on $K_\ep$. On the other side, $\pi_n u_m$ is progressively measurable for any
$m\in\NN$ and $n\in\NN$.  Thus, we can choose $\{\bar
u_n^m,m\in{\mathbb{N}}\}$ to be progressively measurable.  Since $\bar u_n^m\to \bar u_n$ as $m\to\infty$ weakly in
$\DD(0,T;H^\gamma _2(\RR^d ))$ and $\bar u_n\to \bar u^ \ast$ as $n\to\infty$
weakly also in $\DD(0,T;H^\gamma _2(\RR^d ))$, it follows that $\bar u^ \ast$ is a
limit  in $\DD(0,T;H^ \gamma _2(\RR^d ))$ of some progressively
measurable step functions. In particular, $\bar u^ \ast$ is also
progressively measurable, i.e.\ b.) holds.
\end{proof}


\paragraph{\bf Step V:}
In the last step we will show that the process $\bar u^\ast$ is indeed a mild solution to \eqref{itoeqn}.
In particular, we will show that for any $t\in[0,T]$ the identity \eqref{eq_1} is satisfied.
But before, we will state the following proposition.

\begin{proposition}\label{converges}
Let $H$ be a Hilbert space $(\CT(t))_{t\ge 0}$ be a group on $H$ and $\{\xi_m:m\in\NN\}\subset \CM^2 ([0,T];H)$ a sequence of progressively measurable processes such that $\EE \int_0^ T |\xi_m(s)-\xi(s)|_H^ 2\, ds\to 0$ and there exists a constant $C>0$ such that $\EE \int_0^ T |\xi_m(s)|_{H_1}^ 2\, ds\le C$, $m\in\NN$, $H_1\hookrightarrow H$ compactly. Then the process
$$
[0,T]\ni t\mapsto \mathfrak{S}(\xi_m)(t)= \int_0^t \CT(t-s)\xi_m(s)\tilde \eta(dz,ds) 
$$
converges to $[0,T]\ni t=\int_0^t \CT(t-s)\xi(s)\,\tilde \eta(dz,ds)$ in $\DD(0,T;H)$.
\end{proposition}

\begin{proof}
The proof follows by Theorem 7.8 \cite[Chapter 3.7, p.\ 131]{ethier} and the fact that
the set of the laws $\{\mathfrak{S}(\xi_m): m\in\NN\}$ tight on $\DD([0,T];H)$ is.

\end{proof}
\medskip

In order to show the identity \eqref{eq_1}  fix $t\in[0,T]$.
The Lebesgue's Dominated Convergence Theorem gives
$$
\bar \EE \phi( \bar u^\ast (t)) =\bar \EE \phi\lk( \lim_{m\to\infty} \bar u_m(t)\rk) =  \lim_{m\to\infty}\bar \EE \phi\lk(  \bar u_m(t)\rk),
$$
for any $\phi \in \CC_b(L^2(\RR^d))$.
Since  for any $m\in\NN$, the pairs $(\bar u_m,\bar \eta_m)$ have the same laws as
the random variables $( u_m, \eta_ m)$, we can infer that
$$
\bar \EE \phi( u^\ast (t))  = \lim_{m\to\infty} \EE \phi\lk(  u_m(t)\rk).
$$
Since $u_m$ is a mild solution to \eqref{oben_m},
$$
\bar \EE \phi( u^\ast (t)) = \lim_{m\to\infty} \EE \phi\lk( (\mT u_0)(t) + \mF_F(u_m)(t)  + \mS^m (u_m)(t)+ \mF_z ( u_m)(t) 
 \rk),
$$
where
\DEQS
\mS^m  (\xi)(t) =\int_0^t \int_{Z} \CT(t-s) \xi(s) g(z)\;\tilde \eta_m(dz,ds) ,\quad t\in[0,T].
\EEQS
Again, since  $(\bar u_m,\bar \eta_m)$, $m\in\NN$, have the same laws as
the random variables $( u_m, \eta_ m)$, 
$$
\bar \EE \phi( u^\ast (t))    = \lim_{m\to\infty} \bar \EE \phi\lk( (\mT u_0) (t) + \mF_F( \bar u_m)(t)  + \bar \mS^m ( \bar u_m)(t) + \mF_z (\bar u_m)(t)\rk),
$$
where
\DEQS
\bar \mS^m  ( \bar u_m)(t) =\int_0^t \int_{Z} \CT(t-s)  \bar u_m(s) g(z)\;\tilde {\bar{\eta}}_m(dz,ds) ,\quad t\in[0,T]
\EEQS
and
\DEQS
\bar \mF_z  ( \bar u_m)(t) =\int_0^t \int_{Z} \CT(t-s)  \bar u_m(s) h(z) {{\nu}}_m(dz) \, ds,\quad t\in[0,T].
\EEQS
It remains to show that $\bar u_m\to \bar u^\ast$ $\PP$-a.s.\ in $\DD(0,T;H^\gamma_2(\RR^d))$ implies
\DEQS
\lqq{ \lim_{m\to\infty} \bar \EE \phi\lk(  (\mT u_0) (t) + \mF_F( \bar u_m)(t)  + \bar \mS^m ( \bar u_m)(t) + \mF_z (\bar u_m)(t) \rk)} &&\\
& = &
 \bar \EE \phi\lk(  (\mT u_0) (t) + \mF_F(\bar u^\ast)(t)  + \mS_{\ast} (\bar u^\ast)(t)  + \mF_z ( \bar u^ \ast) (t)\rk)
\EEQS
for any $\phi \in  \CC_b(L^ 2(\RR^d))$, 
 where
\DEQS
 \mS  _{\ast}(\bar u^\ast)(t) =\int_0^t \int_{Z} \CT(t-s)\bar u^\ast(s) g(z)\;\tilde {{\eta^ \ast}}(dz,ds) ,\quad t\in[0,T].
\EEQS

\medskip

First, we will show that
$ \mF_F( \bar u_m)\to  \mF_F( \bar u^\ast)$ in $L^2(\RR^ d)$. 
Since we have for any $\gamma<1$
$$
\bar u_m\longrightarrow  \bar u^\ast  \quad \mbox{in}\quad \DD(0,T; H_2^\gamma(\RR^d)) \quad \mbox{as} \quad m\to\infty,
$$
and $\alpha<{d+2\over d-2}$,
we know by Sobolev embedding Theorems
$$
\bar u_m\longrightarrow  \bar u^\ast  \quad \mbox{in}\quad \DD(0,T; L^{\alpha+1}(\RR^d)) \quad \mbox{as} \quad m\to\infty.
$$
Since $F:L^{\alpha+1}(\RR^d )\to L^{\alpha +1\over \alpha}(\RR^d)$ is continuous,
we obtain
$$
F(\bar u_m)\longrightarrow F(\bar u^\ast)  \quad \mbox{in}\quad L^\infty(0,T; L^{\alpha +1\over \alpha}(\RR^d)) \quad \mbox{as} \quad m\to\infty.
$$
Since $L^\infty(0,T)\hookrightarrow L^{q'}(0,T)$,
$$
F(\bar u_m)\longrightarrow F( \bar u^\ast)  \quad \mbox{in}\quad L^{q'}(0,T; L^{\alpha +1\over \alpha}(\RR^d)) \quad \mbox{as} \quad m\to\infty.
$$
By the Strichartz estimate we get
$$
 \mF_F( \bar u_m) \longrightarrow  \mF_F( \bar u^\ast)  \quad \mbox{in}\quad L^\infty(0,T; L^{2}(\RR^d)) \quad \mbox{as} \quad m\to\infty.
$$

\medskip

Next, we will investigate the third summand. By Step V  we have $\PP$-a.s.
$$
\bar u_m \longrightarrow \bar u^\ast \quad \mbox{in}\quad \DD(0,T;H^{\gamma }_2(\RR^d )), \quad \mbox{as} \quad m\to\infty.
$$
By Hypothesis \ref{hypogeneral}, we can infer by Theorem \ref{RS3} that
$$
\bar u_m g(z)\longrightarrow \bar u^\ast g(z)\quad \mbox{in}\quad \DD(0,T;H^{\gamma }_2(\RR^d )) \quad \mbox{as} \quad m\to\infty.
$$
Since $\int_Z 1_{|z|>\ep_m}|z|^2 \nu(dz) \longrightarrow \int_Z |z|^2 \nu(dz)$, it follows $1_{|z|>\ep_m} u_m z  \longrightarrow u^\ast z $
in $\CM^2 (0,T,L^ 2(\RR^d ))$ and therefore,
$$
1_{|z|> \ep_m}\bar u_m z\longrightarrow\bar  u^\ast z\quad \mbox{in}\quad \CM^2 (0,T;L^ 2(\RR^d )) \quad \mbox{as} \quad m\to\infty.
$$
Therefore it follows by Proposition \ref{converges}
$$
\bar \mS^m(\bar u_m z)\longrightarrow  \mS_\ast(\bar u^\ast z)\quad \mbox{in}\quad \DD(0,T;L^ 2(\RR^d )) \quad \mbox{as} \quad m\to\infty.
$$
For the last term, one shows with the same kind of arguments that  $\bar \mF_z  ( \bar u_m)$ converges to $\bar \mF_z (\bar u^\ast)$, as $m\to \infty$,
in $L^{\infty}(0,T;L^2(\RR^d))$.
Hence, we have for all $t\in[0,T]$,
$$
\bar \EE \phi(\bar  u^\ast (t))    = \bar \EE \phi\lk( (\mT u_0) (t) + \mF_F(\bar   u^ \ast)(t)  + \mS_\ast ( \bar u^\ast)(t) + \mF_z (\bar u^\ast)(t)\rk),
$$
Since $\bar u^ \ast$ and $\bar u_m$ belong $\PP$-a.s.\ to $\DD(0,T;L^ 2 (\RR^d))$, it follows by Theorem 7.8 \cite[p. 131]{ethier} that $\bar u_m\to \bar u^ \ast$ in $\DD(0,T;L^ 2 (\RR^d))$.
Hence,  $\bar u^\ast$ is indeed a solution to \eqref{itoeqn}.
\end{proof}

\appendix

\section{Multiplication}

In the section we recall some well known facts concerning Nemytskii operators, which are necessary to prove our main result.
Most of the content is taken from Runs and Sickel  \cite{runst}.

First let us introduce some functions spaces. Let $1\le p\le \infty$ and $m\in\NN_0$, then $W^m_p$ denotes the Sobolev spaces defined by
$$
W_p^m(\RR^d ) :=\lk\{ f\in L^p(\RR^d): |f|_{W^p_m} = |f|_{L^p}+ \sum_{|\alpha|\le m} |D^\alpha f|_{L^p}<\infty\rk\}.
$$
Let $1\le p < \infty$ and $s>0$, $s\not \in \NN$,  then $W^s_p$ denotes the Slobodeckij spaces defined by
\DEQS
\lqq{ W_p^s(\RR^d ) :=\lk\{ f\in L^p(\RR^d): |f|_{W^p_m}^p = |f|_{W^{[s]}_p}^p\rk.
} && \\
&&\lk. {} +\sum_{|\alpha|=[s]}  \int_{\RR^d}  \int_{\RR^d} {|D^\alpha f(x) - D^\alpha f(y)|^p \over |x-y|^{n + (s-[s])p}}\, dx\,dy<\infty\rk\}.
\EEQS
Let $1\le p < \infty$ and $s\in\RR$,  then $H^s_p$ denotes the Bessel Potential spaces or Sobolev spaces of fractional order defined by
\DEQS
H_p^s(\RR^d ) :=\lk\{ f\in L^p(\RR^d): |f|_{H^s_p} = |\CF^{-1} ( (1+\xi^2 )^{\frac s2} \CF f) |_{L^p}
<\infty\rk\}.
\EEQS
Finally, let us introduce the Triebel Lizorkin spaces $F_{p,q}^s(\RR^d)$ and the Besov spaces $B_{p,q}^s(\RR^d)$ by

Let us shortly recall some known identities. The proof can be found e.g.in  \cite{triebel}.
\begin{itemize}
     \item $L_p(\RR^d ) = F_{p,2}^0(\RR^d ) $ for $1<p<\infty$,
     \item $W_p^m (\RR^d ) = F_{p,2}^m(\RR^d ) $ for $1<p<\infty$, $m\in\NN$,
     \item $W_p^s (\RR^d ) = F_{p,p}^s(\RR^d )=B_{p,s}^s(\RR^d )  $ for $1<p<\infty$, $s>0$, $s\not\in\NN$,
     \item $H_p^s (\RR^d ) = F_{p,2}^s(\RR^d ) $  for $1<p<\infty$, $s\in\RR$.
      \end{itemize}

In order to treat the nonlinearity, we list here some useful results.
Assume $s_1<0<s_2$.

\begin{theorem}\label{RS2}(see \cite[p.\ 229]{runst})
Assume $s=s_1\le s_2$, $s_1+s_2>d\cdot\max(0,\frac 1p-1)$, and $q\ge \max(q_1,q_2)$. Then
\begin{itemize}
  \item if $s_2>s_1$, then $F_{p,q_1}^{s_1}(\RR^d) \cdot B_{\infty,q_2}^{s_2} (\RR^d)\hookrightarrow F_{p,q_1}^{s_1}(\RR^d)$;
  \item if $s_1=s_2$ then $F_{p,q_1}^{s_1} (\RR^d)\cdot B_{\infty,q_2}^{s_1}(\RR^d) \hookrightarrow F_{p,q}^{s_1}(\RR^d)$
\end{itemize}
\end{theorem}

\begin{theorem}\label{RS3}(see \cite[p.\ 238]{runst})
Let $s>0$,
$$
\frac 1 {r_1} =\frac 1d \lk( \frac d{p_1} -s\rk) >0, \quad \mbox{ and } \quad
\frac 1 {r_2} =\frac 1d \lk( \frac d{p_2} -s\rk) >0,
$$
and
$$
\frac 1 {r_1}+\frac 1 {r_2} =\frac 1r =\frac 1d \lk( \frac dp -s\rk) <1.
$$
Then
$$
F_{p_1,q_1}^s (\RR^d ) \cdot F_{p_2,q_2}^s (\RR^d ) \hookrightarrow F_{p,q}^s (\RR^d ) ,
$$
iff $$ \max( q_1,q_2)\le q \le \infty.
$$
In addition,
$$
B_{p_1,q_1}^s (\RR^d ) \cdot B_{p_2,q_2}^s (\RR^d ) \hookrightarrow F_{p,q}^s (\RR^d ) ,
$$
iff $$ \max( q_1,q_2)\le q \le \infty, \quad \mbox{ and } \quad 0<q_1\le r_1,\quad  0<q_2\le r_2.
$$
\end{theorem}

\section{A Tightness criteria in $\mathbb{D}([0,T];Y)$}
Let $Y$ be a separable and complete metric space and $T>0$. The
space $\mathbb{D}([0,T];Y)$ denotes the space of all right continuous
functions $x:[0,T]\to Y$ with left limits.
The space of continuous functions is usually equipped with the
uniform topology. But, since $\mathbb{D}([0,T];Y)$ is complete but not
separable in the uniform topology, we equip $\mathbb{D}([0,T];Y)$ with the
Skorohod topology in which $\mathbb{D}([0,T];Y)$ is both separable and
complete. For more information about Skorokhod space and topology
we refer to Billingsley's book \cite{billingsley} or Ethier and
Kurtz \cite{ethier}. In this appendix we only state the
following tightness criterion which is necessary for our work. For
this we denote by ${\mathcal{P}}\left( \mathbb{D}([0,T];Y)\right)$ the space of Borel
probability measures on $\mathbb{D}([0,T];Y)$.

\begin{corollary}\label{comp-2}
Let $\{ x_n:n\in {\mathbb{N}}\}$ be a sequence of  c\`{a}dl\`{a}g processes,
each of the process defined on a probability space $
(\Omega_n,{\mathcal{F}}_n,\mathbb{P}_n)$.
Then the sequence of laws of  $\{ x_n:n\in {\mathbb{N}}\}$ 
is tight on  $\mathbb{D}([0,T];Y)$ if
\begin{enumerate}
\item   there exists a space $Y_1$, $Y_1 \hookrightarrow Y$ compactly, such that 
such that
$$ \EE^ n\left| x_n(t)\right|_{Y_1}^ r\le C 
,\, \forall n\in{\mathbb{N}} ;
$$
\item there exist two constants $c>0$ and $\gamma> 0$ and a real
number $r>0$ such that for all $\theta>0$, $t\in[0,T-\theta]$, and
$ n\ge 0$
$$
 {\mathbb{E}}_n\sup_{t\le s\le t+\theta} |x_n(t)-x_n(s )|_Y ^ r\le c\, \theta ^ \gamma.
$$
\end{enumerate}
\end{corollary}
\begin{proof}
The inequality \ref{comp-2}-(a) and the
Chebyscheff  inequality gives the necessary conditions for the compact containment condition. Next, comparing with  \cite[Theorem 7.2, p.\ 128]{ethier}.
%
Now fix $t_1\le t\le t_2$. Then \begin{eqnarray*}\lefteqn{
 \mathbb{P} _ n\left( |x_n(t)-x_n(t_1)|_Y\ge \lambda, \, |x_n(t)-x_n(t_2)|_Y\ge \lambda\right) }
&&\\&& \le \mathbb{P}_n\left( \sup_{t_1\le s\le t_2} |x_n(s)-x_n(t_1)|_Y\ge
\lambda\right). \end{eqnarray*} Estimating the RHS by the Chebyshev
inequality and using inequality \ref{comp-2}-(b) leads to inequality
the second condition in \cite[Theorem 7.2, p.\ 128]{ethier}. 
Thus the assertion follows.
\end{proof}

\section{Compactness methods}

\begin{lemma}\label{compact}
Let $1\le p<\infty$.
A set $\CA\subset L^p (\RR^d )$ is compact, if
\begin{enumerate}
  \item there exists a $\gamma>0$ such that $\CA$ is bounded in $L_{\gamma}^p(\RR^d)$;
  \item there exist numbers  $\delta\ge 0$, $s>0$, a sequence $R_n$ with $R_n\le R_{n+1}$ and $R_n\uparrow\infty$,
   and a constant $C>0$ such that
  $$\sup_{n\ge 1} \frac 1 {R_n^\delta} \lk| v 1_{B(R_n)}\rk|_{H^s_2 (B(R_n))}\le C,\quad v\in \CA.\footnote{The function $1_{B(R_n)}$ denotes the indicator function of $B(R_n)$ and $B(R_n)=\{ x\in\RR ^d:|x|\le R_n\}$.}
$$
\end{enumerate}
\end{lemma}

\begin{proof}[Proof of Lemma \ref{compact}:]
Let $\{ u_n:n\in\NN\}\subset \CA$ be a sequence. Then, we have to show that there exists a subsequence $\{n_k:k\in\NN\}$ and a $u^ \ast$ such that
$$ u_{n_k} \longrightarrow u^ \ast , \quad \mbox{for} \quad k\to\infty \quad \mbox{in} \quad L^ p(\RR^d),
$$
or, that there exists a subsequence $\{n_k:k\in\NN\}$  which is a  Cauchy sequence in $L^ p(\RR^d)$.
The existence of a unique limit $u ^\ast$ follows.

In the next steps we will construct a subsequence, and then show that this subsequence is a Cauchy sequence.
Since $\{ u_n:n\in\NN\}$ is bounded in $L^ p (\RR^d)$, there exists a subsequence $\{n_k:k\in\NN\}$ and $a^ \ast\in[0,\infty)$ such that
$ |u_{n_k}|^p_{L^ p (\RR^d)}\to a^\ast$ as $k\to\infty$. If $a^\ast=0$, then $\{ u_{n_k}:k\in\NN\}$ converges to zero and we are done.
Let us assume $a^\ast>0$.

Next, we can assume, that
$
|u_{n_k}|^p_{L^p}\in (\frac 12 a ^\ast ,\frac 32 a ^\ast )$ 
for
all $k\in\NN$. 
Let $\{R_m:m\in\NN\}$ a sequences in $\RR_0^+$ with $R_m \uparrow\infty$, such that
$$
\int_{\RR ^d \setminus B(R_m)} |u_n(x)| ^p \, dx \le \frac 14 \, 2 ^{-m} |u_{n}| ^p_{L ^p},\quad \forall \, n\in\NN,\,\, m\in\NN.
$$
Due to the first condition on $\CA$ such a sequence exists. In particular, it follows by an application of the Chebyscheff inequality and the fact that $|u_{n_k}|^p_{L^p}\ge \frac  12 a ^\ast$.
Since $\CA$ is bounded in $H ^s_p(B(R_1))$ and
$H ^s_p(B(R_1))\hookrightarrow L ^p(B(R_1))$ compactly,
there exists  a subsequence $\{ n ^1 _{k}:k\in\NN\}$ of $\{ n_{k}:k\in\NN\}$ such that
$\{
u_{n ^1_{k}}:k\in\NN\}$ is a Cauchy sequence in $L ^p(B(R_1))$.

Again, since $\CA$ is bounded in $H ^\gamma_p(B(R_2))$ and
$H ^\gamma_p(B(R_2))\hookrightarrow L ^p(B(R_2))$ compactly,
there exists  a subsequence $\{ n ^2_{k}:k\in\NN\}$ of $\{ n ^{1}_{k}:k\in\NN\}$ such that
$\{
u_{n ^2_{k}}:k\in\NN\}$ is a Cauchy sequence in $L ^p( B(R_2))$.

Proceeding in this way we obtain subsequences $\{n ^m_k:k\in\NN\}$, $m\in\NN$,
such that
\begin{itemize}
  \item $\{ n_k ^1:k\in\NN\}\supset \{ n_k ^2:k\in\NN\}\supset \{ n_k ^3:k\in\NN\}\supset \ldots $;
  \item for each fixed $m\in\NN$, $\{ u_{n ^m_l}:l\in\NN\}$ is a Cauchy sequence in $L ^p(B(R_m))$;
  \item for each fixed $m\in\NN$ we have $\int_{\RR ^d \setminus B(R_m)} |u_n(x)| ^p \, dx \le 2 ^{-m} |u_{n}| ^p_{L ^p},\, \forall n\in\NN$;
%
\end{itemize}
Let $\{\tilde n_k:k\in\NN\}$ be the diagonal sequence defined for $k\in\NN$ by $\tilde n_k:=n_k ^k$.
Now, we claim that $\{u_{\tilde n_k}:k\in\NN\}$ is a Cauchy sequence in $L ^p(\RR ^d)$.
In order to show it, fix $\ep>0$. The task is now to find an index $k_1\in\NN$ such that
$$\int_{\RR ^d } |u_{n_k}-u_{n_l}| ^p \, dx \le\ep,\quad \forall\, l,k\ge k_1.
$$
Let
$m\in\NN$ be the smallest integer such that $m\ge \ln_2(a^\ast)-\ln_2(\frac{\ep}{6})$.
Since $\{  u _{n^m _k}:k\in\NN \}$ is a Cauchy sequence in $L ^p(B(R_{m}))$ and
$\{ \tilde u _{n_k}:k\ge m\}\subset \{  u _{n^m _k}:k\in\NN \}$, it follows that $\{ \tilde u _{n_k}:k\ge m\}$  is a Cauchy sequence in $L ^p(B(R_{m}))$. Therefore, there exists a $k_1\ge m$ such that
$$
\int_{B(R_m) } |u_{n_k}-u_{n_l}| ^p \, dx \le \frac \ep 2, \quad k,l\ge k_1.
$$
Then we have for any $k,l\ge k_1$
\DEQS
\int_{\RR ^d } |u_{n_k}-u_{n_l}| ^p \, dx \le \int_{ B(R_{m}) } |u_{n_k}-u_{n_l}| ^p \, dx+
 \int_{\RR ^d \setminus B(R_{m})} |u_{n_k}-u_{n_l}| ^p \, dx.
\EEQS
By the choice of $m$ we have
\DEQS
 \int_{\RR ^d \setminus  B(R_{m})} |u_{n_k}-u_{n_l}| ^p \, dx &\le&    \int_{\RR ^d \setminus  B(R_{m})} |u_{n_k}| ^p
 +  \int_{\RR ^d \setminus B(R_{m})} |u_{n_l}| ^p \, dx
 \\
\le 2 ^{-m}  |u_{n_k}| ^p+ 2 ^{-m}   |u_{n_l}| ^p& \le &  2 ^ {-m+1} \frac32 a^\ast
\le \ep/2
.
\EEQS
Collecting altogether,  we get
\DEQS
\int_{\RR ^d } |u_{n_k}-u_{n_l}| ^p \, dx \le
\frac \ep 2 + \frac \ep 2 =\ep,
\EEQS
and the assertion follows.

\end{proof}

\end{document}